\documentclass[a4paper,twoside,10pt]{amsart}
\usepackage{amsmath,amsfonts,amssymb,amsthm,mathrsfs}
\usepackage[T1]{fontenc}
\usepackage{ae,aecompl}
\usepackage{calligra,color,enumerate,graphicx}
\usepackage[a4paper]{geometry}
\usepackage[colorlinks=true,pdfstartview={XYZ null null 1.00},pdfview=FitH,citecolor=blue,urlcolor=customgreen]{hyperref}
\usepackage{caption, subcaption}
\newtheorem{theorem}{Theorem}[section]
\newtheorem{lemma}{Lemma}[section]

\newtheorem{condition}{Condition}[section]
\numberwithin{equation}{section}
\def\im{\mathrm{i}}
\def\d{\mathrm{d}}
\def\e{\mathrm{e}}
\def\O{\mathcal{O}}
\def\eps{\varepsilon}
\def\arcsec{\operatorname{arcsec}}
\def\Ai{\mathscr{A}\!\!\text{\calligra i}\,}
\definecolor{customgreen}{rgb}{0.0, 0.5, 0.0}

\author[G. Nemes]{Gerg\H{o} Nemes}
\email{nemes.gergo@renyi.hu}
\address{Alfr\'ed R\'enyi Institute of Mathematics, Re\'altanoda utca 13--15, Budapest H-1053, Hungary}


\begin{document}

\title[Borel summability of WKB solutions]{On the Borel summability of WKB solutions\\ of certain Schr\"odinger-type differential equations}

\begin{abstract} A class of Schr\"odinger-type second-order linear differential equations with a large parameter $u$ is considered. Analytic solutions of this type of equations can be described via (divergent) formal series in descending powers of $u$. These formal series solutions are called the WKB solutions. We show that under mild conditions on the potential function of the equation, the WKB solutions are Borel summable with respect to the parameter $u$ in large, unbounded domains of the independent variable. It is established that the formal series expansions are the asymptotic expansions, uniform with respect to the independent variable, of the Borel re-summed solutions and we supply computable bounds on their error terms. In addition, it is proved that the WKB solutions can be expressed using factorial series in the parameter, and that these expansions converge in half-planes, uniformly with respect to the independent variable. We illustrate our theory by application to a radial Schr\"odinger equation associated with the problem of a rotating harmonic oscillator and to the Bessel equation.
\end{abstract} 
\maketitle

\section{Introduction and main results}

In this paper, we study Schr\"odinger-type differential equations of the form
\begin{equation}\label{Eq0}
\frac{\d^2 w(u,z)}{\d z^2} = u^2 f(u,z)w(u,z),
\end{equation}
where $u$ is a real or complex parameter, $z$ lies in some domain $\mathbf{D}$, bounded or otherwise, of a Riemann surface, and the potential function $f(u,z)$ is an analytic function of $z$ having the form
\[
f(u,z) = f_0 (z) + \frac{f_1 (z)}{u} + \frac{f_2 (z)}{u^2}.
\]
It is assumed that the coefficient functions $f_n(z)$ are analytic functions of $z$ in the domain $\mathbf{D}$ and are independent of $u$. Furthermore, we suppose that $f_0(z)$ does not vanish in $\mathbf{D}$. We are interested in asymptotic expansions of analytic solutions of \eqref{Eq0} when the parameter $u$ becomes large. It is convenient to work in terms of the transformed variables $\xi$ and $W(u,\xi)$ given by\footnote{This transformation is often called the Liouville transformation.}
\begin{equation}\label{Liouville}
\xi =\xi(z)  = \int_{z_0}^z f_0^{1/2} (t)\d t ,\quad w(u,z) = f_0^{ - 1/4} (z)W(u,\xi).
\end{equation}
The path of integration, except perhaps its endpoint $z_0$, must lie entirely within $\mathbf{D}$. In the case that $z_0$ is a boundary point of $\mathbf{D}$, it is assumed that the integral converges as $t\to z_0$ along the contour of integration. The transformation \eqref{Liouville} maps $\mathbf{D}$ on a domain $\mathbf{G}$, say. In terms of the new variables, equation \eqref{Eq0} becomes
\begin{equation}\label{Eq}
\frac{ \d ^2 W(u,\xi)}{ \d \xi ^2 } = \left(u^2 +u\phi(\xi)+\psi (\xi )\right)W(u,\xi),
\end{equation}
where
\[
\phi (\xi ) = \frac{f_1 (z)}{f_0 (z)}\quad \text{and} \quad \psi (\xi ) = \frac{f_2 (z)}{f_0 (z)} - \frac{1}{f_0^{3/4} (z)}\frac{\d^2}{\d z^2}\frac{1}{f_0^{1/4} (z)}
\]
are holomorphic functions of $\xi$ in $\mathbf{G}$. It can be verified by direct substitution that the equation \eqref{Eq} has formal solutions of the form
\begin{equation}\label{WKB1}
W^\pm  (u,\xi) = \exp\! \left(  \pm u\xi  \pm \frac{1}{2}\int_{}^\xi \phi (t)\d t \right)\left( 1 + \sum\limits_{n = 1}^\infty \frac{\mathsf{A}_n^ \pm  (\xi)}{u^n} \right),
\end{equation}
where the coefficients $\mathsf{A}_n^ \pm  (\xi )$, which are analytic functions of $\xi$ in $\mathbf{G}$, can be determined recursively by
\begin{equation}\label{wkbcoeffrec}
\mathsf{A}_{n + 1}^ \pm  (\xi ) = - \frac{1}{2}\phi (\xi )\mathsf{A}_n^ \pm  (\xi )  \mp \frac{1}{2}\frac{\d \mathsf{A}_n^ \pm  (\xi )}{\d\xi }  \mp \frac{1}{2}\int_{}^\xi  \left( \frac{1}{4}\phi ^2 (t) \mp \frac{1}{2}\phi '(t) - \psi (t) \right)\mathsf{A}_n^ \pm  (t)\d t, 
\end{equation}
with the convention $\mathsf{A}_0^\pm (\xi ) = 1$. The constants of integration in \eqref{WKB1} and \eqref{wkbcoeffrec} are arbitrary. An alternative method for the evaluation of the coefficients $\mathsf{A}_n^\pm (\xi )$, which avoids nested integrations, can be found in Appendix \ref{Appendix}.

The formal solutions \eqref{WKB1} are usually referred to as WKB solutions, in honour of physicists Wentzel \cite{Wentzel1926}, Kramers \cite{Kramers1926} and Brillouin \cite{Brillouin1926}, who independently discovered these solutions in a quantum mechanical context in 1926. To be strictly accurate, the WKB solutions were discussed earlier by Carlini \cite{Carlini1817} in 1817, by Liouville \cite{Liouville1837} and Green \cite{Green1837} in 1837, by Strutt (Lord Rayleigh) \cite{Rayleigh1912} in 1912, and again by Jeffreys \cite{Jeffreys1924} in 1924. The reader is referred to the book of Dingle \cite[Ch. XIII]{Dingle1973} for a historical discussion and critical survey.

In general, the series in \eqref{WKB1} diverge, and the most that can be established is that in certain subregions of $\mathbf{G}$ they provide uniform asymptotic expansions of solutions of the differential equation \eqref{Eq}. Olver \cite[Ch. 10, \S9]{Olver1997} demonstrated the asymptotic nature of the expansions \eqref{WKB1} via the construction of explicit error bounds. In this paper we take a different approach and study the WKB solutions \eqref{WKB1} from the point of view of Borel summability. Our main aim is to show that the expansions \eqref{WKB1} are Borel summable in well-defined subdomains $\Gamma^\pm$ of $\mathbf{G}$, provided that $\phi(\xi)$ and $\psi(\xi)$ satisfy certain mild requirements. We shall construct exact, analytic solutions $W^\pm (u,\xi)$ of the differential equation \eqref{Eq} in terms of Laplace transforms (with respect to the parameter $u$) of some associated functions, called the Borel transforms of the WKB solutions. The formal expansions in \eqref{WKB1} will then arise as asymptotic series of these exact solutions. By carefully analysing the Borel transforms, we will derive new, computable bounds for the remainder terms of the asymptotic expansions. In addition, it will be shown that the WKB solutions can be expressed using factorial series in the parameter $u$, and that these expansions converge for $\Re u>0$, uniformly in the independent variable $\xi$.

The present work is a first step towards the global analysis of the WKB solutions \eqref{WKB1}. Our main result demonstrates that the WKB solutions are Borel summable provided that one stays away from Stokes curves emerging either from transition points (zeros or singularities) of $f_0(z)$ or from singular points of $f_1(z)$ and $f_2(z)$ (see Remark (i) following Theorem \ref{thm1}). In particular, the problem of global connection formulae for WKB solutions across Stokes curves is not studied in this paper. A possible approach to tackle such problems is discussed briefly in Section \ref{Section7}.

Over the past several decades, there has been increasing interest in the study of Borel summability of asymptotic series, and this field is closely linked with exponential asymptotics (see, for example, \cite{OldeDaalhuis2003}). Indeed, there have been extensive developments in summability of singular variable asymptotic expansions, with fewer general results currently existing for the more complicated singular parameter case. The relevance of Borel summability to the WKB analysis was first clearly observed by Bender and Wu \cite{Bender1969}. In the seminal work \cite{Voros1983}, Voros showed how to analyse Borel re-summed WKB solutions and demonstrated the relationship between the singular points of the Borel transforms and the global connection formulae. Dunster et al. \cite{Dunster1993} studied differential equations of the type \eqref{Eq} in the case that $\phi(\xi) \equiv 0$ (or, equivalently, $f_1(z) \equiv 0$). Under suitable conditions on $\psi(\xi)$, they established the Borel summability of WKB solutions away from Stokes curves and derived their convergent factorial series expansions. Their analysis was based on a method of N\o rlund. Kamimoto and Koike \cite{Kamimoto2011} considered the equation \eqref{Eq0} in the special case that $f(u,z)=f(z)$ is a rational function and is independent of $u$, and proved connection formulae for WKB solutions by demonstrating the Borel summability of certain WKB theoretic transformation series introduced originally by Aoki et al. \cite{Aoki1991}. Their results hold in appropriate neighbourhoods of Stokes curves issuing from either a simple zero or a simple pole of $f(z)$ (see also \cite{Sasaki2013}). For further contributions, see, for example, \cite{Aoki2009,Aoki2019,Bodine2002,Delabaere1999,Dillinger1993,Dunster2001,Dunster2004,Iwaki2014,Iwaki2016,Takahashi2019,Takei2017} and the references therein.

Before stating our results, we introduce the necessary assumptions, notation and definitions. Hereinafter, we will suppose that the functions $\phi(\xi)$ and $\psi(\xi)$ fulfil at least one of the following two conditions.

\begin{condition}\label{cond1} The functions $\phi(\xi)$ and $\psi(\xi)$ are analytic in a domain $\Delta \subset \mathbf{G}$, which has the property that there exist positive constants $c$ and $\rho$ independent of $u$ and $\xi$, such that
\[
\left| \phi (\xi) \right| \le \frac{c}{1 + \left| \xi  \right|^{1+\rho} },\quad \left| \psi (\xi) \right| \le \frac{c}{1 + \left| \xi  \right|^{1+\rho} }
\]
for all $\xi \in \Delta$.
\end{condition}

\begin{condition}\label{condv2} The functions $\phi (\xi)$ and $\psi (\xi)$ are analytic in a domain $\Delta \subset \mathbf{G}$, which has the property that there exist positive constants $c$ and $\rho$ independent of $u$ and $\xi$, such that
\[
\left| \phi (\xi ) \right| \le \frac{c}{1 + \left| \xi  \right|^{1/2+\rho} },\quad \left| \phi' (\xi ) \right| \le \frac{c}{1 + \left| \xi  \right|^{1+\rho} },\quad \left| \psi (\xi ) \right| \le \frac{c}{1 + \left| \xi  \right|^{1+\rho} }
\]
for all $\xi \in \Delta$.
\end{condition}

Often these conditions can be brought about by a suitable normalisation of the differential equation \eqref{Eq}. (See the second example in Section \ref{Section6}, below.)

Let $d>0$. We denote by $\Gamma^\pm (d)$ any (non-empty) subdomains of $\Delta$ that satisfy the following two requirements:
\begin{enumerate}[(i)]
	\item The distance between each point of $\Gamma^\pm (d)$ and each boundary point of $\Delta$ has lower bound $d$ (which is to be chosen independently of $u$).\footnote{It will be assumed throughout the paper that $d$ is always chosen so that none of the domains $\Gamma^\pm (d)$ is empty.}
	\item If $\xi \in \Gamma^\pm (d)$, then $\xi \mp x \in \Gamma^\pm (d)$ for all $x>0$.
\end{enumerate}
Note that condition (ii) requires $\Delta$ to contain at least one infinite strip that is parallel to the real axis. In particular, $\Delta$ (and hence $\mathbf{G}$) has to be unbounded. We shall often integrate along half-lines that are parallel to the real axis. For any given $w\in \Gamma^+(d)$, $\mathscr{P}(-\infty,w)$ denotes the half-line that runs from $ - \infty  + \im\Im w$ to $w$. Observe that by condition (ii), $\mathscr{P}(-\infty,w)$ lies entirely in $\Gamma^+(d)$. Similarly, for any $w\in \Gamma^- (d)$, $\mathscr{P}(w, + \infty )$ will denote the half-line emanating from $w$ and passing to $+\infty + \im\Im w$. If the orientations of these paths are reversed, we will adopt the notation $\mathscr{P}(w, -\infty)$ and $\mathscr{P}( + \infty ,w)$, respectively.

We seek solutions of the differential equation \eqref{Eq} of the form
\begin{equation}\label{WKBsol}
W^\pm  (u,\xi ) = \exp\! \left( \pm u\xi  \pm \frac{1}{2}\int_{\alpha^\pm}^\xi \phi (t)\d t \right)\left( 1 + \eta^{\pm}(u,\xi) \right)
\end{equation}
where $\Re u>0$, $\xi \in \Gamma^\pm (d)$, $\alpha^\pm \in \mathbf{G}$ are fixed, and $\eta ^ \pm  (u,\xi )$ satisfy the limit conditions
\begin{equation}\label{limitreq}
\mathop {\lim }\limits_{\Re \xi  \to  \mp \infty } \eta ^ \pm  (u,\xi )=0.
\end{equation}
In \eqref{WKBsol}, the contours of integration have to lie entirely within $\mathbf{G}$ and can be infinite provided the integrals converge. It is readily seen that if such solutions exist, they must be unique. We would like our solutions to posses asymptotic expansions of the form
\begin{equation}\label{WKBexp}
\eta^{\pm}(u,\xi) \sim \sum\limits_{n = 1}^\infty \frac{\mathsf{A}_n^ \pm  (\xi )}{u^n},
\end{equation}
as $|u|\to +\infty$, uniformly with respect to $\xi \in \Gamma^\pm (d)$. The coefficients $\mathsf{A}_n^ \pm  (\xi )$ must satisfy recurrence relations of the type \eqref{wkbcoeffrec} with a suitable choice of integration constants. Note that \eqref{limitreq} implies $\mathop {\lim }\nolimits_{\Re \xi  \to  \mp \infty } \mathsf{A}_n^ \pm  (\xi )=0$.
 We will show in Section \ref{Section2} that if $\mathsf{A}_0^ \pm  (\xi ) = 1$ and
\begin{equation}\label{Arec}
\mathsf{A}_{n + 1}^ \pm  (\xi ) = - \frac{1}{2}\phi (\xi )\mathsf{A}_n^ \pm  (\xi )  \mp \frac{1}{2}\frac{\d \mathsf{A}_n^ \pm  (\xi )}{\d\xi }  \mp \frac{1}{2}\int_{\mathscr{P}( \mp \infty ,\xi )} \left( \frac{1}{4}\phi ^2 (t) \mp \frac{1}{2}\phi '(t) - \psi (t) \right)\mathsf{A}_n^ \pm  (t)\d t   
\end{equation}
for $n\geq 0$ and $\xi \in \Gamma^\pm (d)$, then the requirements $\mathop {\lim }\nolimits_{\Re \xi  \to  \mp \infty } \mathsf{A}_n^ \pm  (\xi )=0$ are fulfilled. The $\mathsf{A}_n^ \pm  (\xi )$'s defined in this way will be the coefficients in the asymptotic expansions \eqref{WKBexp} of our solutions \eqref{WKBsol}, and throughout the rest of the paper, unless stated otherwise, $\mathsf{A}_n^ \pm  (\xi )$ will always refer to the coefficients specified by the above recurrence relation.

Finally, we define, for any $r>0$,
\[
B(r) = \left\{ t:\left| t \right| < r \right\} \quad \text{and} \quad U(r) = \left\{ t:\Re t > 0,\, \left| \Im t \right| < r \right\} \cup B(r).
\]
Thus, $U(r)$ consists of all points whose distance from the positive real axis is strictly less than $r$. We are now in a position to formulate our main result.

\begin{theorem}\label{thm1} Let $d>0$. If Condition \ref{cond1} or \ref{condv2} hold then the differential equation \eqref{Eq} has solutions of the form \eqref{WKBsol}, where $\eta^\pm (u,\xi )$ are analytic functions in $\left\{ u:\Re u > 0 \right\} \times \Gamma ^ \pm  (d)$, satisfy \eqref{limitreq} and can be represented in the form
\begin{equation}\label{Laplacetrans}
\eta^\pm (u,\xi )  = \int_0^{ + \infty } \e^{ - ut} F^\pm (t,\xi )\d t.
\end{equation}
The Borel transforms $F^\pm (t,\xi)$ are analytic functions in $U(2d) \times \Gamma^\pm (d)$ with convergent power series expansions
\[
F^\pm (t,\xi) = \sum\limits_{n = 0}^\infty \frac{\mathsf{A}_{n+1}^\pm(\xi)}{n!}t^n,
\]
valid when $(t,\xi) \in B(2d) \times \Gamma^\pm (d)$. In addition, for each $\sigma>0$, $0<r<d$, there exist positive numbers $K^\pm = K^\pm(\sigma,r)$ independent of $t$ and $\xi$, such that
\begin{equation}\label{Borelbound}
\left| F^\pm  (t,\xi ) \right| \le K^ \pm  \frac{\e^{\sigma \Re t}}{\max (1,\mp\operatorname{sgn}(\Re \xi )\left|\Re \xi\right|^\rho  )},
\end{equation}
when $(t,\xi) \in U(2r) \times \Gamma^\pm (d)$. Finally, for each $\sigma>0$, 
\begin{equation}\label{WKBexp2}
\eta^{\pm}(u,\xi) \sim \sum\limits_{n = 1}^\infty \frac{\mathsf{A}_n^ \pm  (\xi )}{u^n}
\end{equation}
as $|u|\to +\infty$ in the open half-plane $\Re u >\sigma$, uniformly with respect to $\xi \in \Gamma^\pm (d)$. In particular, these asymptotic expansions hold in any closed subsector of $\left| \arg u\right| < \frac{\pi}{2}$.
\end{theorem}

\paragraph{\emph{Remarks.}}
\begin{enumerate}[(i)] 
	\item In order to gain a better understanding of the main result, consider the following special case. Assume that $\phi(\xi)$ and $\psi(\xi)$ are (possibly multivalued) holomorphic functions in the $\xi$-plane, save for a countable number of singularities (which may include branch points), located at $\xi=\xi_k$, $k\in \mathbb{Z}_{\geq 0}$. Let us focus on the Borel summability of the WKB solution $W^+ (u,\xi)$. We introduce brunch cuts from the singularities $\xi_k$ to $ + \infty  + \im \Im \xi _k$. Suppose that $\phi(\xi)$ and $\psi(\xi)$ are analytic in the resulting domain and satisfy the estimates posed in Condition \ref{cond1} or \ref{condv2}. Then we can set
\[
\Gamma ^ +  (d) = \mathbb{C} \setminus \bigcup\limits_{k \ge 1} \left\{ \xi:\left| \xi - \xi _k \right| \le d \right\} \cup \left\{ \xi :\Re \xi  \ge \Re \xi _k,\, \left| \Im (\xi - \xi _k ) \right| \le d \right\} .
\]
The pre-image in the $z$-plane of each of the brunch cuts is a Stokes curve emanating either from a transition point of $f_0(z)$ or from a singularity of $f_1(z)$ or $f_2(z)$ (see, e.g., \cite{Kawai2005,Koike2000}). Hence, under the present circumstances, Theorem \ref{thm1} guarantees the Borel summability of the WKB solution $W^+ (u,\xi)$ away from Stokes curves. An analogous statement holds for the other solution $W^- (u,\xi)$.
\item Using Theorem \ref{thm1}, it is possible to establish the Borel summability of WKB solutions whose $n$th asymptotic expansion coefficients evaluate to $a_n^\pm$ at some fixed $\beta^\pm \in \Gamma^\pm (d)$ with prescribed sequences $\left\{ a_n^\pm \right\}_{n \in \mathbb{Z}_{> 0}}$ of complex numbers. Indeed, assume that the formal series 
\[
\sum\limits_{n = 1}^\infty \frac{a_n^\pm}{u^n}
\]
are Borel summable in $U(2d)$ for some $d>0$, provided $\Re u>\sigma' \geq 0$. Let us denote by $F^\pm (t)$ the respective Borel transforms. Now if $W^\pm (u,\xi)$ denote the WKB solutions given by \eqref{WKBsol} and Theorem \ref{thm1}, then
\[
W^\pm (u,\xi,\beta^\pm) = W^\pm (u,\xi)\left(1 + \int_0^{+\infty} \e^{-ut} F^\pm(t,\beta^\pm )\d t\right)^{-1} \left(1+\int_0^{+\infty} \e^{-ut} F^\pm(t)\d t\right)
\]
are also solutions of the differential equation \eqref{Eq} for all sufficiently large values of $|u|$. The right-hand side can be re-expressed in the form
\[
\exp \!\left(  \pm u\xi  \pm \frac{1}{2}\int_{\alpha ^ \pm  }^\xi \phi (t)\d t\right)\left(1 + \int_0^{+\infty} \e^{-ut} F^\pm(t,\xi,\beta^\pm )\d t\right),
\]
where, for each $0<r<d$, the functions $F^\pm (t,\xi,\beta^\pm )$ are analytic in $(t,\xi) \in U(2r) \times \Gamma^\pm (d)$, and the Laplace transforms converge for $\Re u > K^ \pm + \sigma' + \sigma$, where $\sigma$ is any positive number and $K^\pm = K^\pm(\sigma,r)$ are the corresponding constants given in Theorem \ref{thm1} (cf. \cite[Ch. 5, Theorem 5.55]{Mitschi2016}). By appealing to Watson's lemma \cite[Ch. 4, \S3]{Olver1997}, the Laplace transforms admit asymptotic series in inverse powers of the large parameter $u$ and, due to construction, the coefficients of these series meet the above requirements.

\item Theorem \ref{thm1} generalises the results obtained by Dunster et al. \cite{Dunster1993} in the following two respects: (1) The equation \eqref{Eq} is more general than that considered in their work as it contains the (not necessarily zero) term $u \phi(\xi)$. (2) They make stronger assumptions on $\psi(\xi)$, namely
\[
\left| \psi (\xi ) \right| \le \frac{c}{1 + \left| \xi  \right|^2},\quad \xi  \in \Delta 
\]
and
\begin{equation}\label{psiinequality}
\int_{\mathscr{P}( \mp \infty ,\xi )} \left| \psi (t)\d t \right| \le \frac{c}{\max (1, \mp \Re \xi )},\quad \xi  \in \Gamma ^ \pm (d).
\end{equation}
To prove Theorem \ref{thm1}, we shall also need an inequality analogous to \eqref{psiinequality} but this inequality will be derived as a direct consequence of Condition \ref{cond1} or \ref{condv2}.
\end{enumerate}

In Theorem \ref{thm2} below, we give explicit bounds for the error terms of the asymptotic expansions \eqref{WKBexp2}. To state the theorem, we require some notation. For any $\sigma>0$ and $\xi \in \Gamma ^ \pm  (d)$, we define
\begin{equation}\label{Gdef}
V^{\pm}(\sigma ,\xi ) = \begin{cases}\displaystyle
      \frac{6}{\sigma}\int_{\mathscr{P}(\mp \infty,\xi)}  \left( \frac{1}{4}\left| \phi ^2 (t) \right| + \left| \frac{1}{2}\phi '(t) \pm \psi (t) \right| \right)|\d t| & \text{if } \phi (t) \not\equiv 0, \\ \displaystyle
     \frac{1}{\sigma}\int_{\mathscr{P}(\mp \infty,\xi)} \left| \psi (t) \d t\right| & \text{if } \phi (t) \equiv 0,
\end{cases}
\end{equation}
and for any $0<r<d$, we set
\begin{equation}\label{Cdef}
C^ \pm = C^ \pm (r)  = 2\mathop {\sup }\limits_{\substack{|t| = 2r \\ \xi  \in \Gamma^\pm (d)}} (\max (1,\mp\operatorname{sgn}(\Re \xi )\left| \Re \xi \right|^\rho  )\left| F^\pm  (t,\xi ) \right|) ,
\end{equation}
where $F^\pm  (t,\xi )$ are the Borel transforms appearing in Theorem \ref{thm1}. For an upper bound on the quantities $C^\pm$, see Section \ref{Section4}. With these notation, we have the following theorem.

\begin{theorem}\label{thm2} Assume that Condition \ref{cond1} or \ref{condv2} hold. Then for each positive integer $N$, the functions $\eta^{\pm}(u,\xi)$ given in Theorem \ref{thm1} can be expressed in the form
\begin{equation}\label{truncated}
\eta^{\pm}(u,\xi) = \sum\limits_{n = 1}^{N-1} \frac{\mathsf{A}_n^ \pm  (\xi )}{u^n} + R_N^ \pm  (u,\xi ),
\end{equation}
where, for any $\sigma>0$ and $0<r<d$, the remainder terms $R_N^ \pm  (u,\xi )$ satisfy
\[
\left| R_N^ \pm  (u,\xi ) \right| \le C^ \pm \left( \frac{2r}{N} + \frac{1}{\Re u - \sigma }\right) \frac{\e^{V^{\pm}(\sigma ,\xi )} }{\max (1,\mp\operatorname{sgn}(\Re \xi )\left| \Re \xi  \right|^\rho  )}\frac{N!}{(2r\left| u \right|)^N }
\]
provided $\Re u>\sigma$ and $\xi \in \Gamma^\pm (d)$.
\end{theorem}

To obtain sharp bounds for $R_N^ \pm  (u,\xi )$, we may choose the parameters $\sigma$ and $r$ for each $u$ and $\xi$ separately. For example, one can take $\sigma = \frac{1}{2}\Re u$ and $r$ to be any number that is less than the distance of $\xi$ to the boundary of $\Delta$.

In the following theorem, we provide convergent alternatives to the asymptotic expansions \eqref{WKBexp2}. The coefficients in these expansions depend on an additional (positive) parameter $\omega$ and can be generated by the recurrence relations
\begin{gather}\label{Brec}
\begin{split}
\mathsf{B}_{n + 1}^ \pm  (\omega ,\xi ) =\; & \left( (n - 1)\omega  - \frac{1}{2}\phi (\xi ) \right)\mathsf{B}_n^ \pm  (\omega ,\xi )  \mp \frac{1}{2}\frac{\d \mathsf{B}_n^ \pm  (\omega ,\xi )}{\d\xi }
\\ & \mp \frac{1}{2}\int_{\mathscr{P}( \mp \infty ,\xi )} \left( \frac{1}{4}\phi ^2 (t) \mp \frac{1}{2}\phi '(t) - \psi (t) \right)\mathsf{B}_n^ \pm  (\omega ,t)\d t,
\end{split}
\end{gather}
for $n\geq 1$ with $\mathsf{B}_1^ \pm  (\omega ,\xi ) = \mathsf{A}_1^ \pm  (\xi )$ and $\xi \in \Gamma^\pm (d)$. Note that the $\mathsf{B}_{n}^ \pm(\omega ,\xi )$'s are polynomials in $\omega$ of degree $n-1$ and are analytic functions of $\xi$ in $\Gamma^\pm (d)$. A different expression for these coefficients involving the Stirling numbers of the first kind is given by \eqref{Balt}.

\begin{theorem}\label{thm3} Let $d>0$ and $\omega > \frac{\pi }{4d}$. Assume that Condition \ref{cond1} or \ref{condv2} hold. Then the functions $\eta^{\pm}(u,\xi)$ given in Theorem \ref{thm1} have the factorial expansions
\begin{equation}\label{factorialseries}
\eta^{\pm}(u,\xi) = \sum\limits_{n = 0}^\infty  \frac{\mathsf{B}_{n + 1}^ \pm  (\omega ,\xi )}{u(u + \omega ) \cdots (u + n\omega )} ,
\end{equation}
which are absolutely convergent for $\Re u>\sigma>0$ ($\sigma$ arbitrary), uniformly for $\xi \in \Gamma^\pm (d)$. The errors committed by truncating the series \eqref{factorialseries} after $N\geq 0$ terms do not exceed in absolute value
\begin{equation}\label{fseriesbound}
 \frac{C^ \pm}{2^{\sigma /\omega } }\frac{\Gamma (\Re u/\omega  - 1)}{\omega  - \sigma }\frac{\e^{V^{\pm}(\sigma ,\xi )} }{\max (1,\mp\operatorname{sgn}(\Re \xi )\left| \Re \xi \right|^\rho)}\left(N+\frac{1}{2}\right)^{1 - \Re u/\omega },
\end{equation}
provided that $\Re u >\omega>\sigma>0$ and $\xi \in \Gamma^\pm (d)$. The quantities $C^\pm$ are defined by \eqref{Cdef} with $r = \frac{\pi}{4\omega}$.
\end{theorem}

The remaining part of the paper is structured as follows. The proof of Theorem \ref{thm1} consists of two steps. In Section \ref{Section2}, we prove that the formal Borel transforms of the WKB expansions converge in a neighbourhood of the origin. This is shown via some Gevrey-type estimates for the coefficients $\mathsf{A}_n^\pm(\xi)$. The second step involves analytically continuing the Borel transforms in a strip containing the positive real axis. We construct the continuation as a solution of an integral equation in Section \ref{Section3}, following ideas of Dunster et al. \cite{Dunster1993}. In Section \ref{Section4}, we provide the proof of the error bounds given in Theorem \ref{thm2}. The convergent factorial expansions and the corresponding error bounds stated in Theorem \ref{thm3} are proved in Section \ref{Section5}. In Section \ref{Section6}, we illustrate our results by application to a radial Schr\"odinger equation associated with the problem of a rotating harmonic oscillator and to the Bessel equation. The paper concludes with a discussion in Section \ref{Section7}.

\section{Pre-Borel summability of the WKB solutions}\label{Section2}

In this section, we prove the pre-Borel summability of the WKB solutions, i.e., we show that the formal Borel transforms of the WKB expansions \eqref{WKBexp} are analytic near the origin.

\begin{theorem}\label{pbs} Assume Condition \ref{cond1} or \ref{condv2}. Then the power series
\begin{equation}\label{gpowser}
\sum\limits_{n = 0}^\infty \frac{\mathsf{A}_{n+1}^ \pm  (\xi )}{n!}t^n
\end{equation}
are convergent when $t\in B(2d)$ and $\xi \in \Gamma^\pm (d)$, and define analytic functions $G^\pm (t,\xi )$ in $B(2d) \times \Gamma^\pm (d)$. Moreover, when multiplied by $\max (1, \mp\operatorname{sgn} (\Re \xi )\left| \Re \xi \right|^\rho  )$, the functions $G^\pm  (t,\xi )$ (along with their partial derivatives) are bounded on $B(2r) \times \Gamma^\pm (d)$ for each $0<r<d$.
\end{theorem}

In order to prove Theorem \ref{pbs}, we shall establish a series of lemmata.
 
\begin{lemma}\label{lemma1}
Let $d>0$. Assume Condition \ref{cond1} or \ref{condv2}. Then there exists a positive constant $c_1$, independent of $u$ and $\xi$, such that for any non-negative integer $j$, the following estimates hold:
\begin{equation}\label{diffbound1}
\begin{gathered}
\left| \phi ^{(j)} (\xi ) \right| \le \frac{c_1}{\max (1, \mp\operatorname{\rm sgn} (\Re \xi )\left| \Re \xi \right|^\rho  )}\frac{j!}{d^j},\quad \left| (\phi ^2 (\xi ))^{(j)} \right| \le \frac{c_1}{\max (1, \mp\operatorname{\rm sgn} (\Re \xi )\left| \Re \xi \right|^\rho  )}\frac{j!}{d^j },\\ \left| \psi^{(j)} (\xi ) \right| \le \frac{c_1}{\max (1, \mp\operatorname{\rm sgn} (\Re \xi )\left| \Re \xi \right|^\rho  )}\frac{j!}{d^j}
\end{gathered}
\end{equation}
when $\xi \in \Gamma^{\pm}(d)$.
\end{lemma}

\begin{proof} We shall prove that Lemma \ref{lemma1} holds with $c_1 = \max (c,c^2 )(1 + d)^\rho$. Suppose that $\xi \in \Gamma^{\pm}(d)$ and let $j$ be a non-negative integer. To prove the first inequality in \eqref{diffbound1}, we use Condition \ref{cond1} or \ref{condv2} and Cauchy's formula, to deduce
\begin{equation}\label{dervbound}
\left| \phi^{(j)} (\xi ) \right| = \left| \frac{j!}{2\pi \im}\oint_{\left| t - \xi  \right| = d} \frac{\phi (t)}{(t - \xi )^{j + 1}}\d t \right| \le \frac{c }{2\pi }\frac{j!}{d^{j + 1} }\oint_{\left| t - \xi  \right| = d} \frac{|\d t|}{1 + \left| t \right|^{\mu +\rho} } .
\end{equation}
Here $\mu=1$ or $\mu=1/2$ according to whether we assume Condition \ref{cond1} or \ref{condv2}. Now, when $\left| \xi  \right| \le 1 + d$,
\[
1 + \left| t \right|^{\mu + \rho } \ge 1 \ge \frac{1}{(1 + d)^\rho}\left| \xi  \right|^\rho   \ge \frac{1}{(1 + d)^\rho }\left| \Re \xi \right|^\rho
\]
and when $\left| \xi  \right| \ge 1 + d$,
\[
1 + \left| t \right|^{\mu + \rho }  \ge 1 + \left| \left| \xi  \right| - d \right|^{\mu + \rho }  \ge \left| \left| \xi  \right| - d \right|^\rho   \ge \frac{1}{(1 + d)^\rho}\left| \xi  \right|^\rho   \ge \frac{1}{(1 + d)^\rho}\left| \Re \xi \right|^\rho.
\]
We also have
\[
1 + \left| t \right|^{\mu + \rho }  \ge \frac{1}{(1 + d)^\rho}.
\]
Consequently,
\[
1 + \left| t \right|^{\mu + \rho }  \ge \frac{1}{(1 + d)^\rho}\max (1,\left| \Re \xi \right|^\rho  ) \ge \frac{1}{(1 + d)^\rho}\max (1, \mp\operatorname{\rm sgn} (\Re \xi )\left| \Re \xi \right|^\rho  ),
\]
which, together with \eqref{dervbound}, imply the first inequality in \eqref{diffbound1}. The other bounds in \eqref{diffbound1} can be proved in a similar manner by applying the inequalities
\[
\left| \phi^2 (t) \right| (\le c \left| \phi (t) \right|) \le \frac{c^2}{1+|t|^{\mu+\rho}} ,\quad \left| \psi (t) \right| \le \frac{c}{1+|t|^{1+\rho}},
\]
which hold for all $t\in \Delta$. Here, again, we take $\mu=1$ or $\mu=1/2$ according to whether we assume Condition \ref{cond1} or \ref{condv2}.
\end{proof}

\begin{lemma}\label{lemma2} Let $d>0$. Assume Condition \ref{cond1} or \ref{condv2}. Then there exists a positive constant $c_2$, independent of $u$ and $\xi$, such that the following estimates hold:
\begin{gather}\label{intbound1}
\begin{split}
& \int_{\mathscr{P}( \mp \infty ,\xi )} \left| \phi ^2 (t) \d t\right| \le \frac{c_2}{\max (1, \mp\operatorname{\rm sgn} (\Re \xi )\left| \Re \xi \right|^\rho  )},\\
& \int_{\mathscr{P}(\mp \infty ,\xi )} \left| \phi '(t) \d t\right| \le \frac{c_2}{\max (1, \mp\operatorname{\rm sgn} (\Re \xi )\left| \Re \xi \right|^\rho  )}\frac{1}{d},\\
& \int_{\mathscr{P}( \mp \infty ,\xi )} \left| \psi (t) \d t\right| \le \frac{c_2}{\max (1, \mp\operatorname{\rm sgn} (\Re \xi )\left| \Re \xi \right|^\rho  )}
\end{split}
\end{gather}
when $\xi \in \Gamma^{\pm}(d)$.
\end{lemma}

\begin{proof} We shall show that Lemma \ref{lemma1} holds with $c_2 = 4\max (c,c^2 )(1 + \rho )(1 + (1 + d)^{1 + \rho } )/\rho$. Suppose that $\xi \in \Gamma^{\pm}(d)$. By Condition \ref{cond1}, $|\phi(t)|\leq c$ whenever $t\in\Gamma^{\pm}(d) \subset \Delta$. Therefore, we have the following series of estimates:
\begin{gather}\label{intbound}
\begin{split}
\int_{\mathscr{P}( \mp \infty ,\xi )} \left| \phi ^2 (t)\d t\right|   & \le c \int_{\mathscr{P}( \mp \infty ,\xi )} \left| \phi (t) \d t\right|  = \pm c\int_{ \mp \infty }^{\Re \xi } \left| \phi (s + \im \Im \xi ) \right|\d s \\ & \le  \pm c^2\int_{ \mp \infty }^{\Re \xi } \frac{\d s}{1 + (s^2  + (\Im \xi )^2 )^{(1 + \rho)/2} }  \le  \pm c^2\int_{ \mp \infty }^{\Re \xi } \frac{\d s}{1 + \left| s \right|^{1 + \rho } }.
\end{split}
\end{gather}
If, instead, Condition \ref{condv2} holds, then
\begin{gather}\label{intbound3}
\begin{split}
& \int_{\mathscr{P}( \mp \infty ,\xi )} \left| \phi ^2 (t) \d t\right|  = \pm\int_{ \mp \infty }^{\Re \xi } \left| \phi (s + \im \Im \xi ) \right|^2 \d s  \le  \pm c^2 \int_{ \mp \infty }^{\Re \xi } \frac{\d s}{\left( 1 + (s^2  + (\Im \xi )^2 )^{(1/2 + \rho )/2} \right)^2 }
\\ & \le  \pm c^2 \int_{ \mp \infty }^{\Re \xi } \frac{\d s}{1 + (s^2  + (\Im \xi )^2 )^{1/2 + \rho } }  \le  \pm c^2 \int_{ \mp \infty }^{\Re \xi } \frac{\d s}{1 + \left| s \right|^{1 + 2\rho } } \le  \pm 2c^2 \int_{ \mp \infty }^{\Re \xi } \frac{\d s}{1 + \left| s \right|^{1 + \rho } } .
\end{split}
\end{gather}
Now, when $\mp \Re \xi  > 1$,
\[
 \pm \int_{ \mp \infty }^{\Re \xi } \frac{\d s}{1 + \left| s \right|^{1 + \rho } }  \le  \pm \int_{ \mp \infty }^{\Re \xi } \frac{\d s}{( \mp s)^{1 + \rho } }  = \frac{1}{\rho ( \mp \Re \xi )^\rho  } \leq \frac{2(1 + \rho )}{\rho ( \mp \Re \xi )^\rho  }
\]
and when $\mp \Re \xi  \le 1$,
\[
 \pm \int_{ \mp \infty }^{\Re \xi } \frac{\d s}{1 + \left| s \right|^{1 + \rho } }  \le \int_{ - \infty }^{ + \infty } \frac{\d s}{1 + \left| s \right|^{1 + \rho } }  = 2\int_0^{ + \infty } \frac{\d s}{1 + s^{1 + \rho } }  = 2\Gamma\! \left( \frac{\rho}{1 + \rho } \right)\Gamma\! \left( \frac{2+ \rho}{1 + \rho } \right) \le \frac{2(1 + \rho )}{\rho}.
\]
These two estimates, together with \eqref{intbound} or \eqref{intbound3} and the choice of $c_2$, imply the first inequality in \eqref{intbound1}. The third estimate in \eqref{intbound1} can be obtained in a similar way by employing the inequality
\[
\left| \psi (t) \right| \le \frac{c}{1+|t|^{1+\rho}},
\]
which is valid for $t\in\Gamma^{\pm}(d) \subset \Delta$. To prove the second inequality in \eqref{intbound1}, we can proceed as follows. 
Using Condition \ref{cond1} and Cauchy's formula, we derive
\begin{equation}\label{dervbound2}
\left| \phi' (\xi ) \right| = \left| \frac{1}{2\pi \im}\oint_{\left| t - \xi  \right| = d} \frac{\phi (t)}{(t - \xi )^2}\d t \right| \le \frac{c}{2\pi }\frac{1}{d^2 }\oint_{\left| t - \xi  \right| = d} \frac{|\d t|}{1 + \left| t \right|^{1+\rho} } .
\end{equation}
Now, when $\left| \xi  \right| \le 1 + d$,
\[
1 + \left| t \right|^{1 + \rho }  \ge 1 \ge \frac{1}{1 + (1 + d)^{1 + \rho }}(1 + \left| \xi  \right|^{1 + \rho } )
\]
and when $\left| \xi  \right| \ge 1 + d$,
\[
1 + \left| t \right|^{1 + \rho }  \ge 1 + \left| \left| \xi  \right| - d\right|^{1 + \rho }  \ge 1 + \frac{1}{(1 + d)^{1 + \rho }}\left| \xi  \right|^{1 + \rho }  \ge \frac{1}{1 + (1 + d)^{1 + \rho }}(1 + \left| \xi  \right|^{1 + \rho } ).
\]
Hence, from \eqref{dervbound2}, we can infer that
\begin{equation}\label{phiprimebound}
\left| \phi' (\xi ) \right| \le \frac{c(1 + (1 + d)^{1 + \rho } )}{1 + \left| \xi  \right|^{1 + \rho } }\frac{1}{d}
\end{equation}
for any $\xi \in \Gamma^{\pm}(d)$. Consequently,
\begin{align*}
\int_{\mathscr{P}( \mp \infty ,\xi )} \left| \phi '(t)\d t\right|  & \le \pm c(1 + (1 + d)^{1 + \rho } )\frac{1}{d}\int_{ \mp \infty }^{\Re \xi } \frac{\d s}{1 + (s^2  + (\Im \xi )^2 )^{(1 + \rho )/2} } \\ & \le \frac{2(1 + \rho )c(1 + (1 + d)^{1 + \rho } )}{\rho \max (1, \mp\operatorname{\rm sgn} (\Re \xi )\left| \Re \xi \right|^\rho )}\frac{1}{d}.
\end{align*}
With our choice of the constant $c_2$, this bound implies the second inequality in \eqref{intbound1}. If, instead, Condition \ref{condv2} is assumed, then
\begin{align*}
\int_{\mathscr{P}( \mp \infty ,\xi )} \left| \phi '(t)\d t\right|  & \le  \pm c\int_{ \mp \infty }^{\Re \xi } \frac{\d s}{1 + (s^2  + (\Im \xi )^2 )^{(1 + \rho )/2} }
\\ & \le \frac{2(1 + \rho )c}{\rho \max (1, \mp\operatorname{\rm sgn} (\Re \xi )\left| \Re \xi \right|^\rho )} \le \frac{2(1 + \rho )c(1 + (1 + d)^{1 + \rho } )}{\rho \max (1, \mp\operatorname{\rm sgn} (\Re \xi )\left| \Re \xi \right|^\rho )}\frac{1}{d}.
\end{align*}
By the choice of $c_2$, this bound implies the second inequality in \eqref{intbound1}.
\end{proof}

\begin{lemma}\label{lemma3} If $n$ and $m$ are integers such that $n>0$ and $m\geq 0$, then
\[
\sum\limits_{j = 0}^m \binom{m}{j} j!(n + m - j - 1)!  = \frac{(n + m)!}{n}.
\]
\end{lemma}

\begin{proof} The identity can be proved by re-writing the left-hand side as a telescoping sum (with the convention that $1/( - 1)! = 0$):
\begin{align*}
\sum\limits_{j = 0}^m \binom{m}{j}j!(n + m - j - 1)! & = m!\sum\limits_{j = 0}^m \frac{(n + m - j - 1)!}{(m - j)!}
\\ & = \frac{m!}{n}\sum\limits_{j = 0}^m \left( \frac{(n + m - j)!}{(m - j)!} - \frac{(n + m - j - 1)!}{(m - j - 1)!} \right)  = \frac{(n + m)!}{n}.
\end{align*}
\end{proof}

\begin{proof}[Proof of Theorem \ref{pbs}] We begin by proving that for $\xi \in \Gamma^\pm (d)$, $n\geq 1$ and $m\geq 0$, it holds that
\begin{equation}\label{Abound}
\left| \frac{\d^m \mathsf{A}_n^ \pm  (\xi )}{\d\xi ^m } \right| \le c_3 C_n (d)\frac{1}{2^n}\frac{(n + m - 1)!}{d^{n + m - 1} \max (1, \mp \operatorname{sgn} (\Re \xi )\left| \Re \xi \right|^\rho  )}
\end{equation}
where
\begin{equation}\label{Abound2}
c_3 = 1+ \frac{7}{4}c_2,\quad C_1 (d) = 1+\frac{1}{2}c_1  + \frac{5}{4}c_1 d
\end{equation}
and
\begin{equation}\label{Abound3}
C_{n + 1} (d) = \left( 1 + \frac{c_2}{2}\frac{1}{n} + \left( c_1  + \frac{9}{4}c_2  \right)\frac{d}{n} + \frac{7}{4}c_1 \frac{d^2 }{n^2 } \right)C_n (d)
\end{equation}
for $n\geq 1$. We proceed via induction on $n$. By definition, 
\begin{equation}\label{A1}
\mathsf{A}^\pm _1 (\xi ) =  - \frac{1}{4}\phi (\xi ) \mp \frac{1}{8}\int_{\mathscr{P}( \mp \infty ,\xi )} \phi ^2 (t)\d t  \pm \frac{1}{2}\int_{\mathscr{P}( \mp \infty ,\xi )} \psi (t)\d t .
\end{equation}
Employing Lemmas \ref{lemma1} and \ref{lemma2}, together with the definitions of $c_3$ and $C_1(d)$, we deduce
\[
\left| \mathsf{A}_1^ \pm  (\xi ) \right| \le  \frac{7}{4}c_2 \frac{1}{2}\frac{1}{\max (1, \mp \operatorname{sgn} (\Re \xi )\left| \Re \xi \right|^\rho  )} \le c_3C_1 (d)\frac{1}{2}\frac{1}{\max (1, \mp \operatorname{sgn} (\Re \xi )\left| \Re \xi \right|^\rho  )}.
\]
Now let $m$ be an arbitrary positive integer. Differentiating \eqref{A1} $m$ times yields
\[
\frac{\d^m \mathsf{A}_1^ \pm  (\xi )}{\d\xi ^m }  =  - \frac{1}{4}\phi ^{(m)} (\xi ) \mp \frac{1}{8}\left( \phi ^2 (\xi ) \right)^{(m - 1)} \pm \frac{1}{2}\left( \psi (\xi ) \right)^{(m - 1)} .
\]
By an application of Lemma \ref{lemma1}, and the definitions of $c_3$ and $C_1(d)$, we obtain
\begin{align*}
\left| \frac{\d^m \mathsf{A}_1^ \pm  (\xi )}{\d\xi ^m }  \right| & \le \left( \frac{1}{2}c_1  + \frac{1}{4}c_1 \frac{d}{m} + c_1 \frac{d}{m} \right)\frac{1}{2}\frac{m!}{d^m \max (1, \mp \operatorname{sgn} (\Re \xi )\left| \Re \xi \right|^\rho  )} \\ & \le \left( \frac{1}{2}c_1  + \frac{5}{4}c_1 d \right)\frac{1}{2}\frac{m!}{d^m \max (1, \mp \operatorname{sgn} (\Re \xi )\left|\Re \xi \right|^\rho  )}
\\ & \le c_3C_1 (d)\frac{1}{2}\frac{m!}{d^m \max (1, \mp \operatorname{sgn} (\Re \xi )\left| \Re \xi \right|^\rho  )}.
\end{align*}
Assume that \eqref{Abound} holds for all $\d^m \mathsf{A}_k^ \pm  (\xi ) / \d\xi ^m $ with $1 \leq k \leq n$ and $m\geq 0$. By definition, 
\begin{gather}\label{An}
\begin{split}
\mathsf{A}^\pm_{n + 1} (\xi ) = & - \frac{1}{2}\phi (\xi )\mathsf{A}^\pm_n (\xi ) \mp \frac{1}{2}\frac{\d \mathsf{A}_n^ \pm  (\xi )}{\d\xi } \mp \frac{1}{8}\int_{\mathscr{P}( \mp \infty ,\xi )} \phi ^2 (t)\mathsf{A}^\pm_n (t)\d t \\ &   + \frac{1}{4}\int_{\mathscr{P}( \mp \infty ,\xi )} \phi '(t)\mathsf{A}^\pm_n (t)\d t \pm \frac{1}{2}\int_{\mathscr{P}( \mp \infty ,\xi )} \psi (t)\mathsf{A}^\pm_n (t)\d t.
\end{split}
\end{gather}
Using the induction hypothesis, Lemmas \ref{lemma1} and \ref{lemma2}, and the definition of $C_{n+1}(d)$, we deduce
\begin{align*}
\left| \mathsf{A}^\pm_{n + 1} (\xi ) \right| & \le c_3\left( 1 + \frac{c_2}{2}\frac{1}{n} + \frac{9}{4}c_2 \frac{d}{n} \right)C_n (d)\frac{1}{2^{n + 1} }\frac{n!}{d^n \max (1, \mp \operatorname{sgn} (\Re \xi )\left| \Re \xi \right|^\rho  )}
\\ & \le c_3C_{n + 1} (d)\frac{1}{2^{n + 1} }\frac{n!}{d^n \max (1, \mp \operatorname{sgn} (\Re \xi )\left| \Re \xi \right|^\rho  )} .
\end{align*}
Now let $m$ be an arbitrary positive integer. Differentiating \eqref{An} $m$ times gives
\begin{align*}
\frac{\d^m \mathsf{A}_{n+1}^ \pm  (\xi )}{\d\xi ^m } = & - \frac{1}{2}\sum\limits_{j = 0}^m \binom{m}{j}\phi ^{(j)} (\xi )\frac{\d^{m - j} \mathsf{A}_n^ \pm  (\xi )}{\d \xi ^{m - j} } \mp \frac{1}{2}\frac{\d^{m +1} \mathsf{A}_n^ \pm  (\xi )}{\d \xi ^{m +1} } \\ & \mp \frac{1}{8}\sum\limits_{j = 0}^{m - 1} \binom{m-1}{j}(\phi ^2 (\xi ))^{(j)} \frac{\d^{m - j-1} \mathsf{A}_n^ \pm  (\xi )}{\d \xi ^{m - j-1} }
\\ & + \frac{1}{4}\sum\limits_{j = 0}^{m - 1} \binom{m-1}{j}\phi ^{(j)} (\xi )\frac{\d^{m - j-1} \mathsf{A}_n^ \pm  (\xi )}{\d \xi ^{m - j-1} } \pm \frac{1}{2}\sum\limits_{j = 0}^{m - 1} \binom{m-1}{j}\psi^{(j)} (\xi )\frac{\d^{m - j-1} \mathsf{A}_n^ \pm  (\xi )}{\d \xi ^{m - j-1} }  .
\end{align*}
Using the induction hypothesis, Lemmas \ref{lemma1} and \ref{lemma3}, and the definition of $C_{n+1}(d)$, we obtain
\begin{align*}
\left| \frac{\d^m \mathsf{A}_{n+1}^ \pm  (\xi )}{\d\xi ^m } \right| 
& \le c_3 \left( 1 + c_1 \frac{d}{n} + \frac{7}{4}c_1 \frac{d^2 }{n(n + m)} \right)C_n (d)\frac{1}{2^{n + 1}}\frac{(n + m)!}{d^{n + m} \max (1, \mp \operatorname{sgn} (\Re \xi )\left| \Re \xi \right|^\rho  )}
\\ & \le c_3 \left( 1 + c_1 \frac{d}{n} + \frac{7}{4}c_1 \frac{d^2 }{n^2 } \right)C_n (d)\frac{1}{2^{n + 1}}\frac{(n + m)!}{d^{n + m} \max (1, \mp \operatorname{sgn} (\Re \xi )\left| {\Re \xi } \right|^\rho  )}
\\ & \le c_3 C_{n + 1} (d)\frac{1}{2^{n + 1}}\frac{(n + m)!}{d^{n + m} \max (1, \mp \operatorname{sgn} (\Re \xi )\left| \Re \xi \right|^\rho  )}.
\end{align*}
This completes the proof of the bound \eqref{Abound}.

It is straightforward to show that $C_n (d)$ grows at most polynomially in $n$ (see Section \ref{Section4}). Hence, we obtain from \eqref{Abound} that
\[
\mathop {\lim \sup }\limits_{n \to  + \infty } \left| \frac{\mathsf{A}_{n + 1}^ \pm  (\xi )}{n!} \right|^{1/n}  \le \mathop {\lim \sup }\limits_{n \to  + \infty } \frac{1}{2d}\left( \frac{c_3}{2}C_{n + 1} (d) \right)^{1/n} =\mathop {\lim }\limits_{n \to  + \infty } \frac{1}{2d}\left( \frac{c_3}{2}C_{n + 1} (d) \right)^{1/n}  = \frac{1}{2d}
\]
for any $\xi \in \Gamma^\pm (d)$. Therefore, by the Cauchy--Hadamard theorem, the power series \eqref{gpowser} are convergent when $t\in B(2d)$, and define analytic functions $G^\pm (t,\xi )$ on $B(2d)$ (for each fixed $\xi \in \Gamma^\pm(d)$). From \eqref{Abound} we can also infer that the series \eqref{gpowser} converge uniformly on $\Gamma^\pm(d)$ and hence the functions $G^\pm (t,\xi )$ are analytic with respect to $\xi$ for each fixed $t$. Since the functions $G^\pm(t,\xi)$ are continuous in $(t,\xi)$ and are analytic in
each of their variables, they are analytic functions in $B(2d) \times \Gamma^\pm(d)$.

The estimate \eqref{Abound} also implies that, when multiplied by $\max (1, \mp\operatorname{sgn} (\Re \xi )\left| \Re \xi \right|^\rho  )$, the functions $G^\pm  (t,\xi )$ and their partial derivatives are bounded on $B(2r) \times \Gamma^\pm (d)$ for each $0<r<d$.
\end{proof}

\section{Borel summability of the WKB solutions}\label{Section3}

In this section, we shall show that the functions $\eta^\pm (u,\xi )$ can be represented as Laplace transforms of some associated functions $F^\pm (t,\xi )$ that are analytic in $U(2d) \times \Gamma^\pm(d)$. These associated functions will coincide with $G^\pm (t,\xi )$ of Theorem \ref{pbs} in their common domains of definition. The proof follows closely that in \cite{Dunster1993}: we make a Laplace ``ansatz'' to derive a pair of partial differential equations for $F^\pm (t,\xi )$ with certain boundary conditions. Thus, we seek $F^\pm (t,\xi )$ such that
\[
\eta^\pm (u,\xi )  = \int_0^{ + \infty } \e^{ - ut} F^\pm (t,\xi )\d t,
\]
where $\Re u>0$ and $\xi \in \Gamma^{\pm}(d)$. By partial integration one easily finds that if $F^\pm (t,\xi )$ satisfy
\begin{equation}\label{feq}
 F^\pm_{\xi \xi } (t,\xi )  \pm 2F^\pm_{\xi t} (t,\xi ) = \mp \phi (\xi )F^\pm_\xi  (t,\xi ) + \left(   - \frac{1}{4}\phi ^2 (\xi ) \mp \frac{1}{2}\phi '(\xi ) + \psi (\xi ) \right)F^\pm(t,\xi )
\end{equation}
and the conditions
\begin{equation}\label{boundary}
\pm2F^\pm_\xi  (0,\xi ) = \pm2\frac{\d\mathsf{A}^\pm _1 (\xi )}{\d \xi} =   - \frac{1}{4}\phi^2 (\xi ) \mp \frac{1}{2}\phi'  (\xi ) + \psi (\xi ),
\end{equation}
then $\eta ^ \pm  (u,\xi )$ satisfy
\[
\eta _{\xi \xi }^ \pm  (u,\xi ) \pm (2u + \phi (\xi ))\eta _\xi ^ \pm  (u,\xi ) = \left( - \frac{1}{4}\phi ^2 (\xi ) \mp \frac{1}{2}\phi '(\xi ) + \psi (\xi ) \right)\left( 1 + \eta ^ \pm  (u,\xi ) \right),
\]
and hence \eqref{WKBsol} satisfy the equation \eqref{Eq}. We seek solutions of \eqref{feq} in $U(2d) \times \Gamma^\pm(d)$ such that
\begin{equation}\label{flimit}
\mathop {\lim }\limits_{\Re \xi  \to  \mp \infty } F^\pm  (t,\xi ) = 0
\end{equation}
for all $t \in U(2d)$. This is consistent with the requirement \eqref{limitreq}. Using \eqref{Arec} and Theorem \ref{pbs}, one can show that $G^\pm(t,\xi)$ are solutions of \eqref{feq} on $B(2d) \times \Gamma^\pm  (d)$. In order to show that $G^{\pm}(t,\xi)$ can be continued analytically to the whole of $U(2d) \times \Gamma^\pm(d)$ and are of exponential type in $t$ as $\Re t\to +\infty$ in $U(2d)$, we will transform \eqref{feq} into an integral equation. We define temporarily
\[
\widetilde F^\pm(\tau,w) = F^\pm  \left( \tau,w \pm \tfrac{1}{2}\tau \right).
\]
Then $F^\pm$ satisfy \eqref{feq}, \eqref{boundary} and \eqref{flimit} if and only if $\widetilde F^\pm$ satisfy
\begin{gather}\label{feq2}
\begin{split}
 & \pm 2\widetilde F^\pm_{w\tau} (\tau,w) =  \mp\phi \left( w \pm \tfrac{1}{2}\tau \right)\widetilde F^\pm_w (\tau,w) 
   \\ & + \left(  - \frac{1}{4} \phi ^2 \left( w\pm \tfrac{1}{2}\tau \right) \mp\frac{1}{2}\phi '\left( w \pm \tfrac{1}{2}\tau \right)+ \psi \left( w \pm \tfrac{1}{2}\tau \right) \right)\widetilde F^\pm(\tau,w) 
\end{split}
\end{gather}
and
\begin{equation}\label{initial}
 \pm 2\widetilde F^\pm_w (0,w) = - \frac{1}{4}\phi ^2 (w) \mp \frac{1}{2}\phi '(w) + \psi (w), \quad \mathop {\lim }\limits_{\Re w \to  \mp \infty } \widetilde F^\pm(\tau,w) = 0.
\end{equation}
Now, we integrate the equation \eqref{feq2} in $\tau$ from $\zeta$ to $t$ and in $w$ from $\xi\mp\frac{1}{2}\zeta$ to $\mp\infty$. Applying the limit conditions \eqref{initial} and integrating once by parts, we obtain
\begin{gather}\label{initialinteq}
\begin{split}
& F^ \pm  \left( t,\xi  \mp \tfrac{1}{2}(\zeta  - t) \right) - F^ \pm  (\zeta ,\xi ) =  - \frac{1}{2}\int_\zeta ^t \phi \left( \xi  \mp \tfrac{1}{2}(\zeta  - \tau ) \right)F^\pm  \left( \tau ,\xi  \mp \tfrac{1}{2}(\zeta  - \tau ) \right)\d\tau 
\\ & \mp \frac{1}{2}\int_{\mathscr{P}\left( \xi  \mp \frac{1}{2}\zeta , \mp \infty \right)} \int_\zeta ^t \left( - \frac{1}{4}\phi ^2 \left( w \pm \tfrac{1}{2}\tau \right) \pm \frac{1}{2}\phi '\left( w \pm \tfrac{1}{2}\tau \right) + \psi \left( w \pm \tfrac{1}{2}\tau \right) \right)F^\pm  \left( \tau ,w \pm \tfrac{1}{2}\tau \right)\d\tau \d w .
\end{split}
\end{gather}
Here $\zeta \in B(2d)$, $t\in U(2d)$, and $\xi, \xi  \mp \tfrac{1}{2}(\zeta - t) \in \Gamma^\pm (d)$.

In what follows, we discuss the construction of $F^-$, an analogous argument can be used to construct $F^+$. It is not easy to use the integral equation \eqref{initialinteq} and the initial condition in \eqref{initial} directly to prove the existence of a solution of \eqref{feq} (satisfying \eqref{flimit}). Since we would like our solution to coincide with the solution $G^-(t,\xi)$ defined in Theorem \ref{pbs} when $t\in B(2d)$, and the only restriction on $\xi$ is $\xi \in \Gamma^-(d)$, we express $t$ as $x+\zeta$, where $x \geq 0$, replace $\xi$ by $\xi+\frac{1}{2} x$, and $F^-  \left(\zeta,\xi  + \tfrac{1}{2}x\right)$ by $G^-  \left(\zeta,\xi  + \tfrac{1}{2}x\right)$ to obtain the following linear integral equation
\begin{gather}\label{finalinteq}
\begin{split}
& F^-  (t,\xi ) = G^-  \left( \zeta,\xi  + \tfrac{1}{2}x \right) - \frac{1}{2}\int_0^x \phi \left( \xi  + \tfrac{1}{2}(x - \tau ) \right)F^-  \left( \tau  + \zeta,\xi  + \tfrac{1}{2}(x - \tau ) \right)\d\tau  
\\ & + \frac{1}{2}\int_{\mathscr{P}\left(\xi  + \frac{1}{2}x, + \infty \right)}  \int_0^x \left(   - \frac{1}{4}\phi ^2 \left( w - \tfrac{1}{2}\tau \right) - \frac{1}{2}\phi '\left( w - \tfrac{1}{2}\tau \right) + \psi \left( w - \tfrac{1}{2}\tau \right) \right)F^-\left( \tau  + \zeta,w - \tfrac{1}{2}\tau \right) \d\tau  \d w .
\end{split}
\end{gather}
Note that if $t\in B(2d)$, then $G^-(t,\xi)$ is a solution of \eqref{finalinteq}. Also observe from the definition of $\Gamma^-(d)$ that $\xi \in \Gamma^-(d)$ implies $\xi +\frac{1}{2}x \in \Gamma^-(d)$ for all $x \geq 0$. Equation \eqref{finalinteq} will eventually be used to define $F^-(t,\xi)$.

Let $\sigma$ be an arbitrary (fixed) positive number. Denote by $\mathscr{B}_\sigma$ the complex vector space of continuous functions $H(x,\xi)$ on $\mathbb{R}_{\geq 0} \times \Gamma^-(d)$ such that for each $H \in \mathscr{B}_\sigma$ there exist a constant $K$ (depending only on $\sigma$) such that
\begin{equation}\label{hineq}
\left| H(x,\xi ) \right| \le K\frac{\e^{V^{-}(\sigma ,\xi )} \e^{\sigma x} }{\max (1,\operatorname{sgn} (\Re \xi )\left| \Re \xi \right|^\rho  )},
\end{equation}
where $V^{-}(\sigma ,\xi )$ is defined in \eqref{Gdef}. Let us define the norm $\left\|H\right\|$ of $H(x,\xi)$ to be the infimum of all such $K$ for which the inequality \eqref{hineq} holds. Equipped with this norm, $\mathscr{B}_\sigma$ becomes a Banach space. In the following, we shall use the facts that for fixed $\xi \in \Gamma^-(d)$, $s \mapsto V^{-}(\sigma ,\xi +s)$ is a monotonically decreasing function, $\mathop {\lim }_{s \to  + \infty } V^{-}(\sigma ,\xi  + s) = 0$ and
\[
\frac{\d}{\d s}V^{-}(\sigma ,\xi  + s) =  - \frac{6}{\sigma}\left( \frac{1}{4}\left| \phi ^2 (\xi  + s) \right| + \left| \frac{1}{2} \phi '(\xi  + s) - \psi (\xi  + s) \right| \right).
\]
(We will assume that $\phi(t) \not\equiv 0$. The case of $\phi(t) \equiv 0$ can be handled in a similar manner.)

Consider the operator
\begin{align*}
& \mathcal{L}H(x,\xi) =  - \frac{1}{2}\int_0^x \phi \left( \xi  + \tfrac{1}{2}(x - \tau ) \right)H\left( \tau ,\xi  + \tfrac{1}{2}(x - \tau ) \right)\d\tau  
\\ & + \frac{1}{2}\int_{\mathscr{P}\left(\xi  + \frac{1}{2}x, + \infty \right)}  \int_0^x \left(   - \frac{1}{4}\phi ^2 \left( w - \tfrac{1}{2}\tau \right) - \frac{1}{2}\phi '\left( w - \tfrac{1}{2}\tau \right) + \psi \left( w - \tfrac{1}{2}\tau \right) \right)H\left( \tau ,w - \tfrac{1}{2}\tau \right) \d\tau  \d w ,
\end{align*}
acting on the space $\mathscr{B}_\sigma$. By the Cauchy--Schwarz inequality, we can assert that
\begin{align*}
& \left( \int_0^x \left| \phi \left( \xi  + \tfrac{1}{2}(x - \tau ) \right)H\left( \tau ,\xi  + \tfrac{1}{2}(x - \tau ) \right) \right|\d\tau  \right)^2 
\\ & \le \int_0^x \left| \phi ^2 \left( \xi  + \tfrac{1}{2}(x - \tau ) \right)\right|\left| H\left( \tau ,\xi  + \tfrac{1}{2}(x - \tau ) \right) \right|\d\tau \int_0^x \left| H\left( \tau ,\xi  + \tfrac{1}{2}(x - \tau ) \right) \right|\d\tau 
\\ & \le \int_0^x \left| \phi ^2 \left( \xi  + \tfrac{1}{2}(x - \tau ) \right) \right|\left\| H \right\|\frac{\exp \left( V^{-}\left(\sigma , \xi  + \frac{1}{2}(x - \tau ) \right) \right)\e^{\sigma \tau } }{\max \left( 1,\operatorname{sgn}\left( \Re \xi  + \frac{1}{2}(x - \tau ) \right)\left| \Re \xi  + \frac{1}{2}(x - \tau ) \right|^\rho   \right)}\d\tau \\ & \quad \times \int_0^x \left\| H \right\|\frac{\exp \left( V^{-}\left(\sigma , \xi  + \frac{1}{2}(x - \tau ) \right) \right)\e^{\sigma \tau } }{\max \left( 1,\operatorname{sgn} \left( \Re \xi  + \frac{1}{2}(x - \tau ) \right)\left| \Re \xi  + \frac{1}{2}(x - \tau ) \right|^\rho  \right)}\d\tau  
\\ & \le \frac{2}{3}\left\| H \right\|^2 \frac{\e^{V^{-}(\sigma ,\xi )} \e^{2\sigma x} }{(\max (1,\operatorname{sgn} (\Re \xi )\left| \Re \xi \right|^\rho  ))^2 }\int_0^x \frac{\d}{\d\tau}\exp \left( V^{-}\left( \sigma ,\xi  + \tfrac{1}{2}(x - \tau ) \right)\right)\d\tau 
\\ & \le \frac{2}{3}\left\| H \right\|^2 \frac{\e^{2V^{-}(\sigma ,\xi )} \e^{2\sigma x} }{(\max (1,\operatorname{sgn}(\Re \xi )\left| \Re \xi \right|^\rho  ))^2 }.
\end{align*}
Consequently,
\[
\left| \frac{1}{2}\int_0^x \phi \left( \xi  + \tfrac{1}{2}(x - \tau ) \right)H\left( \tau ,\xi  + \tfrac{1}{2}(x - \tau ) \right)\d\tau  \right| \le \frac{1}{\sqrt 6}\left\| H \right\|\frac{\e^{V^{-}(\sigma ,\xi )} \e^{\sigma x} }{\max (1,\operatorname{sgn}(\Re \xi )\left| \Re \xi  \right|^\rho  )}.
\]
We also have
\begin{align*}
& \left|\frac{1}{2}\int_{\mathscr{P}\left(\xi  + \frac{1}{2}x, + \infty \right)}  \int_0^x \left(   - \frac{1}{4}\phi ^2 \left( w - \tfrac{1}{2}\tau \right) - \frac{1}{2}\phi '\left( w - \tfrac{1}{2}\tau \right) + \psi \left( w - \tfrac{1}{2}\tau \right) \right)H\left( \tau ,w - \tfrac{1}{2}\tau \right) \d\tau  \d w \right|
\\ & \le \frac{1}{2}\int_{\mathscr{P}\left(\xi  + \frac{1}{2}x, + \infty \right)}  \int_0^x \left(  \frac{1}{4}\left| \phi ^2 \left( w - \tfrac{1}{2}\tau \right)\right| + \left|\frac{1}{2} \phi '\left( w - \tfrac{1}{2}\tau \right) - \psi \left( w - \tfrac{1}{2}\tau \right) \right| \right)    \\ & \quad \times \left\| H \right\|\frac{\exp \left( V^{-}\left( \sigma ,w  - \frac{1}{2}\tau \right)\right)\e^{\sigma \tau } }{\max \left( 1,\operatorname{sgn}\left( \Re w - \frac{1}{2}\tau \right)\left| \Re w  - \frac{1}{2}\tau  \right|^\rho  \right)}\d\tau  \d w  
\\ & \le  - \frac{\sigma }{12}\left\| H \right\|\frac{1}{\max (1,\operatorname{sgn} (\Re \xi )\left| \Re \xi \right|^\rho  )}\int_0^x \e^{\sigma \tau } \int_0^{ + \infty } \frac{\d}{\d s}\exp \left( V^{-}\left( \sigma ,\xi  + \tfrac{1}{2}(x - \tau ) + s \right) \right)\d s \d\tau 
\\ & = \frac{\sigma }{12}\left\| H \right\|\frac{1}{\max (1,\operatorname{sgn} (\Re \xi )\left| \Re \xi \right|^\rho  )}\int_0^x \e^{\sigma \tau } \left( \exp \left( V^{-}\left( \sigma ,\xi  + \tfrac{1}{2}(x - \tau ) \right) \right) - 1 \right)\d\tau 
\\ & \le \frac{1}{12}\left\| H \right\|\frac{\e^{V^{-}(\sigma ,\xi )} \e^{\sigma x} }{\max (1,\operatorname{sgn}(\Re \xi )\left| \Re \xi \right|^\rho  )}.
\end{align*}
Thus, $\mathcal{L}$ is a $\mathscr{B}_\sigma \to \mathscr{B}_\sigma$ linear operator whose induced operator norm $\left\|\mathcal{L}\right\| \leq \frac{1}{\sqrt 6} + \frac{1}{12} < \frac{1}{2}$.

Now, for each $\zeta \in B(2d)$, define the operator $\mathcal{T}_\zeta$, acting on the space $\mathscr{B}_\sigma$, via
\[
\mathcal{T}_\zeta  H(x,\xi ) = G^- \left( \zeta ,\xi  + \tfrac{1}{2}x \right) +\mathcal{L}H(x,\xi ).
\]
Since
\[
\left| \max (1,\operatorname{sgn} (\Re \xi )\left| \Re \xi \right|^\rho  )G^-  \left( \zeta ,\xi  + \tfrac{1}{2}x \right) \right| \le \left| \max \left( 1,\operatorname{sgn} \left( \Re \xi  + \tfrac{1}{2}x \right)\left| \Re \xi  + \tfrac{1}{2}x \right|^\rho  \right)G^-  \left( \zeta ,\xi  + \tfrac{1}{2}x \right) \right|
\]
and the right-hand side is bounded (cf. Theorem \ref{pbs}), $\mathcal{T}_\zeta  H(x,\xi )$ belongs to the space $\mathscr{B}_\sigma$. We also have
\begin{align*}
\left\| \mathcal{T}_\zeta  H_1 (x,\xi ) - \mathcal{T}_\zeta H_2 (x,\xi ) \right\| = \left\| \mathcal{L}H_1 (x,\xi ) - \mathcal{L}H_2 (x,\xi ) \right\| & \le \left\| \mathcal{L}\right\|\left\| H_1 (x,\xi ) - H_2 (x,\xi ) \right\| \\ & < \tfrac{1}{2}\left\| H_1 (x,\xi ) - H_2 (x,\xi ) \right\|
\end{align*}
for any $H_1 (x,\xi )$, $H_2 (x,\xi )\in \mathscr{B}_{\sigma}$. Consequently, $\mathcal{T}_\zeta$ is a $\mathscr{B}_{\sigma} \to \mathscr{B}_{\sigma}$ contraction. Therefore, for each $\zeta \in B(2d)$ there is a unique function $F(x,\xi;\zeta)$ in $\mathscr{B}_{\sigma}$ defined on $\mathbb{R}_{\geq 0} \times \Gamma^-(d)$ which satisfies
\begin{equation}\label{fixedpoint}
F(x,\xi;\zeta) = G^- \left( \zeta ,\xi  + \tfrac{1}{2}x \right) +\mathcal{L}F(x,\xi;\zeta).
\end{equation}
To complete the construction of $F^-(t,\xi)$, it remains to show that for $\zeta \in B(2d)$ these functions of $x$ can be combined to yield an analytic solution of \eqref{finalinteq} on $U(2d)\times \Gamma^-(d)$. We first claim that
\begin{equation}\label{fisg}
F(x,\xi;\zeta) = G^-(x+\zeta,\xi)
\end{equation}
when $\xi \in \Gamma^-(d)$, $\zeta \in B(2d)$, $x\geq 0$ and $x+\zeta \in B(2d)$. This follows by uniqueness since each side is a solutions of \eqref{fixedpoint}, and adding the restriction $x+\zeta \in B(2d)$ to the definition of $\mathscr{B}_{\sigma}$ and $\mathcal{L}$ does not increase the norm. Then, for all real numbers $\delta>0$ and $\zeta+\delta \in B(2d)$, we find from \eqref{fixedpoint} (with $x$ replaced by $x+\delta$) and \eqref{fisg} that
\begin{gather}\label{fineq1}
\begin{split}
& F(x + \delta ,\xi ;\zeta ) = G^-  \left( \zeta ,\xi  + \tfrac{1}{2}(x + \delta ) \right) - \frac{1}{2}\int_0^\delta \phi \left( \xi  + \tfrac{1}{2}(x + \delta  - \tau ) \right)G^- \left( \tau  + \zeta ,\xi  + \tfrac{1}{2}(x + \delta  - \tau ) \right)\d\tau 
\\ & + \frac{1}{2}\int_{\mathscr{P}\left(\xi  + \frac{1}{2}(x+\delta), + \infty \right)}  \int_0^\delta  \left(   - \frac{1}{4} \phi ^2 \left( w - \tfrac{1}{2}\tau \right) - \frac{1}{2}\phi '\left( w - \tfrac{1}{2}\tau \right) + \psi \left( w - \tfrac{1}{2}\tau \right) \right)G^-  \left( \tau  + \zeta ,w - \tfrac{1}{2}\tau  \right)\d\tau \d w 
\\ & - \frac{1}{2}\int_\delta ^{x + \delta } \phi \left( \xi  + \tfrac{1}{2}(x + \delta  - \tau ) \right)F\left( \tau ,\xi  + \tfrac{1}{2}(x + \delta  - \tau );\zeta \right)\d\tau 
\\ & + \frac{1}{2}\int_{\mathscr{P}\left(\xi  + \frac{1}{2}(x+\delta), + \infty \right)} \int_\delta ^{x + \delta } \left(   - \frac{1}{4}\phi ^2 \left( w - \tfrac{1}{2}\tau \right) - \frac{1}{2}\phi '\left( w - \tfrac{1}{2}\tau \right) + \psi \left( w - \tfrac{1}{2}\tau \right) \right)F\left( \tau ,w - \tfrac{1}{2}\tau ;\zeta \right)\d\tau  \d w .
\end{split}
\end{gather}
Now, regarding \eqref{finalinteq} as an equation satisfied by $G^-(x+\zeta,\xi)$, and in that equation replacing $x$ by $\delta$ and $\xi$ by $\xi+\frac{1}{2}x$, we see that the first three terms on the right-hand side of \eqref{fineq1} can be replaced by $G^- \left( \zeta  + \delta ,\xi  + \tfrac{1}{2}x \right)$. Then, if we replace $\tau$ by $\tau+\delta$ in the third and fourth integrals and $w$ by $w+\frac{1}{2}\delta$ in the fourth integral on the right-hand side of \eqref{fineq1}, we deduce
\begin{gather}\label{fineq2}
\begin{split}
& F(x + \delta ,\xi ;\zeta ) = G^- \left( \zeta  + \delta ,\xi  + \tfrac{1}{2}x \right) - \frac{1}{2}\int_0^x \phi \left( \xi  + \tfrac{1}{2}(x - \tau ) \right)F\left( \tau  + \delta ,\xi  + \tfrac{1}{2}(x - \tau );\zeta \right)\d\tau 
\\ & + \frac{1}{2}\int_{\mathscr{P}\left(\xi  + \frac{1}{2}x, + \infty \right)}  \int_0^x \left(   - \frac{1}{4}\phi ^2 \left( w - \tfrac{1}{2}\tau \right) - \frac{1}{2}\phi '\left( w - \tfrac{1}{2}\tau \right) + \psi \left( w - \tfrac{1}{2}\tau \right) \right)F\left( \tau  + \delta ,w - \tfrac{1}{2}\tau ;\zeta \right)\d\tau  \d w .
\end{split}  
\end{gather}
Hence, by uniqueness, it follows from \eqref{fixedpoint} and \eqref{fineq2} that for all real $x\geq 0$, $\delta>0$ and $\zeta, \zeta+\delta\in B(2d)$
\begin{equation}\label{finvar}
F(x,\xi;\zeta+\delta)=F(x+\delta,\xi;\zeta),
\end{equation}
since each is a solutions of the same integral equation.

The equality \eqref{finvar} shows that
\[
F^-(x+\zeta,\xi)=F(x,\xi;\zeta)
\]
is well defined for $\xi \in \Gamma^-(d)$, $\zeta \in B(2d)$ and $x\geq 0$. Let $t=x+\zeta$ be a point in $U(2d)$. We regard $x\geq 0$ as fixed and let $\zeta$ varying in a neighbourhood of the origin. Since $\mathcal{L}$ is independent of $\zeta$, and differentiation with respect to $\zeta$ commutes with $\mathcal{L}$, it follows that $F(x,\xi;\zeta)$ is an analytic function of $\zeta$. Hence $F^-(t,\xi)$ is analytic with respect to $t$ in $U(2d)$ (for each fixed $\xi \in \Gamma^-(d)$), with $\partial F^-/\partial t = \partial F/\partial \zeta$.

Since $F^-(t,\xi)$ is also continuous in $(t,\xi)$ and analytic in $\xi$ for each fixed $t$, it follows that $F^-(t,\xi)$ is an analytic function in $U(2d) \times \Gamma^-(d)$.

Now, by Theorem \ref{pbs}, for each $0<r<d$, there exists a number $C^->0$, independent of $\zeta$, $\xi$ and $x$, such that
\begin{equation}\label{festimate1}
2\left| G^-  \left( \zeta ,\xi  + \tfrac{1}{2}x \right) \right| \le \frac{C^-}{\max \left( 1,\operatorname{sgn}\left( \Re \xi  + \frac{1}{2}x \right)\left| \Re \xi  + \frac{1}{2}x \right|^\rho \right)}  \le \frac{C^-}{\max (1,\operatorname{sgn}(\Re \xi )\left| \Re \xi\right|^\rho )},
\end{equation}
when $\zeta \in B(2r)$, $\xi \in \Gamma^-(d)$ and $x\geq 0$. Thus, for each $\zeta \in B(2r)$, $G^-  \left( \zeta ,\xi  + \tfrac{1}{2}x \right) \in \mathscr{B}_\sigma$ and $\left\|G^-  \left( \zeta ,\xi  + \tfrac{1}{2}x \right)\right\|\leq \frac{1}{2} C^-$ for all $\sigma>0$. Therefore, by \eqref{fixedpoint},
\begin{equation}\label{festimate2}
\left\| F(x,\xi ;\zeta ) \right\| \le 2\left\| G^-  \left( \zeta ,\xi  + \tfrac{1}{2}x \right) \right\| \le C^-
\end{equation}
for $\zeta \in B(2r)$, $\xi \in \Gamma^-(d)$ and $x\geq 0$. Upon expressing $t\in U(2r)$ in the form $t=\Re t+\zeta$ with $\zeta \in B(2r)$, we obtain
\begin{equation}\label{festimate3}
\left| F^-  (t,\xi ) \right| \le \left\| F(\Re t,\xi ;\zeta ) \right\|\frac{\e^{V^{-}(\sigma ,\xi )} \e^{\sigma \Re t}}{\max (1,\operatorname{sgn}(\Re \xi )\left| \Re \xi \right|^\rho  )} \le \left( C^- \mathop {\sup }\limits_{\xi  \in \Gamma^- (d)} \e^{V^{-}(\sigma ,\xi )} \right)\frac{\e^{\sigma \Re t} }{\max (1,\operatorname{sgn}(\Re \xi )\left| \Re \xi \right|^\rho  )}
\end{equation}
for all $\xi \in \Gamma^-(d)$. The finiteness of the supremum on the right-hand side is guaranteed by Lemma \ref{lemma2}. This proves \eqref{Borelbound} and thus the convergence of \eqref{Laplacetrans}.

The validity of the asymptotic expansions \eqref{WKBexp2} follows from Theorem \ref{thm2} which we prove in the forthcoming section.

\section{Error bounds}\label{Section4}

In this section, we prove the error bounds given in Theorem \ref{thm2}. Let $N$ be an arbitrary positive integer. Integrating by parts $N$ times in \eqref{Laplacetrans}, we obtain \eqref{truncated} with
\begin{equation}\label{remainderformula}
R_N^ \pm  (u,\xi ) = \left( \left[ \frac{\d^{N - 1} F^\pm  (t,\xi )}{\d t^{N - 1} } \right]_{t = 0}  + \int_0^{ + \infty } \e^{ -ut} \frac{\d^N F^\pm  (t,\xi )}{\d t^N}\d t \right)\frac{1}{u^N} .
\end{equation}
Let $0<r<d$ and $\sigma>0$. From \eqref{festimate1}--\eqref{festimate3} and the corresponding results for $F^+ (t,\xi )$, we can infer that if
\begin{align*}
C^\pm = C^\pm(r) & = 2\mathop {\sup }\limits_{\substack{|t| =2r \\ \xi  \in \Gamma^\pm (d)}} (\max (1,\mp\operatorname{sgn}(\Re \xi )\left| \Re \xi  \right|^\rho  )\left| G^\pm  (t,\xi ) \right|) \\ & = 2\mathop {\sup }\limits_{\substack{|t| =2r \\ \xi  \in \Gamma^\pm (d)}} (\max (1,\mp\operatorname{sgn}(\Re \xi )\left| \Re \xi \right|^\rho  )\left| F^\pm  (t,\xi ) \right|)
\end{align*}
then
\begin{equation}\label{fbound}
\left| F^\pm (t,\xi ) \right| \le C^ \pm  \frac{\e^{V^{\pm}(\sigma ,\xi )} \e^{\sigma \Re t}}{\max (1,\mp\operatorname{sgn}(\Re \xi )\left| \Re \xi \right|^\rho  )},
\end{equation}
for all $(t,\xi) \in U(2r) \times \Gamma^\pm (d)$. Thus, by Cauchy's formula,
\begin{align*}
\left| \frac{\d^N F^\pm  (t,\xi )}{\d t^N } \right|  = \left| \frac{N!}{2\pi \im}\oint_{\left| w-t \right| = 2r} \frac{F^\pm  (w,\xi )}{(w - t)^{N + 1} }\d w  \right|  &  \le \frac{1}{(2r)^{N + 1} }\frac{N!}{2\pi}\oint_{\left| w-t \right| = 2r} \left| F^\pm  (t + (w - t),\xi ) \right|\left| \d w \right| 
\\ &  \le C^ \pm \frac{\e^{V^{\pm}(\sigma ,\xi )} \e^{\sigma t} }{\max (1,\mp\operatorname{sgn}(\Re \xi )\left| \Re \xi \right|^\rho  )}\frac{N!}{(2r)^N }
\end{align*}
for all $t\geq 0$ and $\xi \in \Gamma^\pm (d)$. Employing this estimate in \eqref{remainderformula} yields
\begin{align*}
\left| R_N^ \pm  (u,\xi ) \right| & \le \left( \left| \left[ \frac{\d^{N - 1} F^\pm  (t,\xi )}{\d t^{N - 1} } \right]_{t = 0} \right| + \int_0^{ + \infty } \e^{ - \Re ut} \left| \frac{\d^N F^\pm  (t,\xi )}{\d t^N} \right|\d t \right)\frac{1}{\left| u \right|^N}
\\ & \le C^ \pm  \left( \frac{2r}{N} + \frac{1}{\Re u - \sigma }\right) \frac{\e^{V^{\pm}(\sigma ,\xi )} }{\max (1,\mp\operatorname{sgn}(\Re \xi )\left| \Re \xi  \right|^\rho  )}\frac{N!}{(2r\left| u \right|)^N },
\end{align*}
provided $\Re u>\sigma$.

It is possible to obtain a simple upper bound for the quantities $C^\pm$ by referring to the results in Section \ref{Section2}. From the definitions of $C^\pm$ and $G^\pm(t,\xi)$, and the estimate \eqref{Abound}, we can assert that
\begin{equation}\label{Cupperbound}
C^{\pm} \le c_3 \sum\limits_{n = 0}^\infty  C_{n + 1} (d)\left( \frac{r}{d} \right)^n.
\end{equation}
Further simplification is possible by bounding the $C_n (d)$'s. To this end, we note that $1 + x < \e^x$ for all $x>0$, and
\[
\sum\limits_{k = 1}^{n - 1} \frac{1}{k} < \log n + \gamma ,\quad \sum\limits_{k = 1}^{n - 1} \frac{1}{k^2}  < \sum\limits_{k = 1}^\infty  \frac{1}{k^2}  = \frac{\pi^2}{6}
\]
for all $n\geq 1$ (cf. \cite[Ch. 8, \S3, Ex. 3.3]{Olver1997}). Therefore, using the definitions \eqref{Abound2} and \eqref{Abound3}, we deduce
\begin{align*}
C_n (d) & = \left( 1 + \frac{1}{2}c_1  + \frac{5}{4}c_1 d \right)\prod\limits_{k = 1}^{n - 1} \left( 1 + \frac{c_2}{2}\frac{1}{k} + \left( c_1  + \frac{9}{4}c_2 \right)\frac{d}{k} + \frac{7}{4}c_1 \frac{d^2}{k^2} \right)
\\ & \le \left( 1 + \frac{1}{2}c_1  + \frac{5}{4}c_1 d \right)\exp \left( \gamma \left( \frac{1}{2}c_2  + \left( c_1  + \frac{9}{4}c_2 \right)d \right) + \frac{7\pi ^2}{24}c_1 d^2 \right)n^{c_2 /2 + (c_1  + 9c_2 /4)d}
\end{align*}
with $c_1  = \max (c,c^2 )(1 + d)^\rho$ and $c_2  = 4\max (c,c^2 )(1 + \rho )(1 + (1 + d)^{1 + \rho } )/\rho$. Substituting this estimate into \eqref{Cupperbound} gives a computable upper bound on $C^\pm$.

\section{Convergent factorial series}\label{Section5}

In this section, we prove the convergent factorial expansions and the corresponding error bounds stated in Theorem \ref{thm3}. Let $d>0$ and $\omega >\frac{\pi}{4d}$, and define $r>0$ by $\omega = \frac{\pi}{4 r}$. Denote by $\Lambda (2r)$ the region in the $t$-plane defined by
\[
\Lambda (2r) = \left\{ t:\left| 1 - \e^{ - \pi t/(4r)} \right| < 1 \right\} \subset U(2r)\subset U(2d).
\]
According to Theorem \ref{thm1}, the functions $F^\pm  (t,\xi )$ are analytic in $\Lambda (2r) \times \Gamma^\pm(d)$ and satisfy the estimate \eqref{Borelbound} for any fixed $\sigma>0$. Consequently, from the fundamental theorem on factorial series \cite[Ch. 11, Theorem 46.2]{Wasow1965}, the functions $\eta^{\pm}(u,\xi)$ posses expansions of the form \eqref{factorialseries} which are absolutely convergent for $\Re u>\sigma>0$ ($\sigma$ arbitrary), for each $\xi \in \Gamma^\pm(d)$. It therefore remains to show that the convergence is uniform in $\xi$. To this end, we show that for $\xi \in \Gamma^\pm(d)$, $0 < \sigma < \omega$ and $n\geq 0$, it holds that
\begin{equation}\label{Bestimate}
\mathsf{B}_{n+1}^ \pm  (\omega,\xi ) =\O(1)\omega^n \frac{(n+1)^{\sigma /\omega  - 1}  n!}{\max (1,\mp\operatorname{sgn}(\Re \xi )\left| \Re \xi \right|^\rho  )}
\end{equation}
where the implied constants are independent of $\xi$ and $n$. The coefficients $\mathsf{B}_{n}^ \pm  (\omega,\xi )$ are the following series expansion coefficients
\[
F^\pm  (t,\xi ) = \sum\limits_{n = 0}^\infty \frac{\mathsf{B}_{n + 1}^ \pm  (\omega,\xi )}{\omega^n n!}\left( 1 - \e^{ - \omega t} \right)^n
\]
with $t\in \Lambda(2r)$ and $\xi \in \Gamma^\pm(d)$ (cf. \cite[pp. 325--326]{Wasow1965}). Hence, by Cauchy's integral formula,
\[
\mathsf{B}_{n+1}^ \pm  (\omega,\xi ) = \frac{\omega^{n+1} n! }{2\pi \im}\oint_{(0+)} \frac{F^\pm  (t,\xi )\e^{ - \omega t} }{\left( 1 - \e^{ - \omega t} \right)^{n+1} } \d t
\]
where the path of integration is a small loop that encircles the origin in the positive sense. Next, on nothing the bound \eqref{Borelbound} (and assuming $0 < \sigma < \omega$), we deform the contour of integration by expanding it to the boundary of the domain $\Lambda(2r)$. Then we split the resulting integral into two parts, the first from $t =  + \infty  + \frac{\pi }{2\omega}\im$ to $t =  - \frac{1}{\omega}\log 2$, and the second from $t =  - \frac{1}{\omega}\log 2$ to $t =  + \infty  - \frac{\pi }{2\omega}\im$. On making the change of integration variables
\[
s = \im\log ( 1 - \e^{ - \omega t} ),
\]
we obtain
\begin{gather}\label{Bintegral}
\begin{split}
\mathsf{B}_{n+1}^ \pm  (\omega,\xi ) =\; & \frac{\omega^n n!}{2\pi}\int_0^\pi  F^\pm  ( - (1/\omega )\log (1 - \e^{ - \im s} ),\xi ) \e^{\im ns} \d s 
\\ & + \frac{\omega^{n} n!}{2\pi}\int_{ - \pi }^0 F^\pm  ( - (1/\omega )\log (1 - \e^{ - \im s} ),\xi ) \e^{\im ns} \d s .
\end{split}
\end{gather}
From \eqref{Borelbound} we have
\begin{align*}
\left| F^\pm  ( - (1/\omega )\log (1 - \e^{ - \im s} ),\xi ) \right| & \le \frac{K^ \pm}{\max (1,\mp\operatorname{sgn}(\Re \xi )\left| \Re \xi \right|^\rho  )}\left( 2\left| \sin \left( \frac{s}{2} \right) \right| \right)^{ - \sigma /\omega } \\ & \le \frac{K^ \pm }{\max (1,\mp\operatorname{sgn}(\Re \xi )\left| \Re \xi \right|^\rho  )}\left( \frac{2\left| s \right|}{\pi }\right)^{ - \sigma /\omega } 
\end{align*}
provided $0<|s|<\pi$, with the constants $K^\pm$ being independent of $s$ and $\xi$. Employing these estimates in \eqref{Bintegral} and appealing to Lemma 12.3 and Ex. 12.2 of Olver \cite[pp. 99--100]{Olver1997}, we deduce the desired result \eqref{Bestimate}. Therefore, if $0 < \sigma < \omega$, the estimates \eqref{Bestimate} and the known asymptotics for the ratio of two gamma functions \cite[\href{http://dlmf.nist.gov/5.11.E12}{Eq. 5.11.12}]{DLMF} give
\[
\frac{\mathsf{B}_{n + 1}^\pm (\omega ,\xi )}{u(u + \omega ) \cdots (u + n\omega )}  =  \frac{\O(1)}{\max (1, \mp \operatorname{sgn}(\Re \xi )\left| \Re \xi \right|^\rho  )}\frac{1}{(n + 1)^{(u - \sigma )/\omega  + 1} },
\]
where the implied constants are independent of $\xi$ and $n$. An appeal to the Weierstrass $M$-test establishes that the series \eqref{factorialseries} indeed converge uniformly for $\xi \in \Gamma^\pm(d)$ provided $\Re u>\sigma$.

By the Stirling algorithm \cite[Proposition 2.8]{Delabaere2007}, the coefficients $\mathsf{A}_n^ \pm  (\xi )$ and $\mathsf{B}_n^ \pm  (\omega,\xi)$ are related to each other via
\begin{equation}\label{Balt}
\mathsf{B}_{n + 1}^ \pm  (\omega ,\xi ) = \sum\limits_{k = 0}^n ( - \omega )^{n - k} s(n,k)\mathsf{A}_{k + 1}^ \pm  (\xi )
\end{equation}
for any $n\geq 0$, where $s(n,k)$ denote the Stirling numbers of the first kind \cite[\href{http://dlmf.nist.gov/26.8.i}{\S26.8(i)}]{DLMF}. Combining \eqref{Balt} with \cite[\href{http://dlmf.nist.gov/26.8.E18}{Eq. 26.8.18}]{DLMF} and the recurrence relations \eqref{Arec} confirms \eqref{Brec}.

We conclude this section with the proof of the estimate \eqref{fseriesbound}. First, we note that, by \eqref{fbound},
\[
\left| F^\pm  ( - (1/\omega )\log (1 - \e^{ - \im s} ),\xi ) \right| \le C^ \pm\frac{\e^{V^{\pm}(\sigma,\xi)}}{\max (1,\mp\operatorname{sgn}(\Re \xi )\left| \Re \xi \right|^\rho  )}\left( \frac{2\left| s \right|}{\pi }\right)^{ - \sigma /\omega } 
\]
whenever $0<|s|<\pi$. Substitution of these bounds into \eqref{Bintegral} and a straightforward estimation yield
\begin{equation}\label{Bestimate2}
\left| \mathsf{B}_{n+1}^ \pm  (\omega,\xi ) \right| \le  \frac{C^ \pm}{2^{\sigma /\omega } }\frac{\omega ^{n + 1} }{\omega  - \sigma }\frac{\e^{V^{\pm}(\sigma ,\xi )} n!}{\max (1,\mp\operatorname{sgn}(\Re \xi )\left| \Re \xi \right|^\rho )},
\end{equation}
for $\xi \in \Gamma^\pm(d)$, $0 < \sigma < \omega$ and $n\geq 0$. It can readily be shown, using the beta integral \cite[\href{http://dlmf.nist.gov/5.12.E1}{Eq. 5.12.1}]{DLMF}, that
\begin{equation}\label{facestimate}
\frac{1}{\left| u(u + \omega ) \cdots (u + n\omega ) \right|} \leq \frac{\Gamma (\Re u/\omega )}{\omega ^{n + 1} \Gamma (\Re u/\omega  + n + 1)}.
\end{equation}
Thus, from \eqref{Bestimate2} and \eqref{facestimate}, we have
\begin{align*}
\left| \sum\limits_{n = N}^\infty  \frac{\mathsf{B}_{n + 1}^ \pm  (\omega ,\xi )}{u(u+\omega)\cdots(u+n\omega)} \right|  
 & \le \frac{C^ \pm}{2^{\sigma /\omega } }\frac{\Gamma (\Re u/\omega )}{\omega  - \sigma }\frac{\e^{V^{\pm}(\sigma ,\xi )} }{\max (1,\mp\operatorname{sgn}(\Re \xi )\left| \Re \xi \right|^\rho)}\sum\limits_{n = N}^\infty  \frac{n!}{\Gamma (\Re u/\omega +n + 1)} 
\\ & \le \frac{C^ \pm}{2^{\sigma /\omega } }\frac{\Gamma (\Re u/\omega )}{\omega  - \sigma }\frac{\e^{V^{\pm}(\sigma ,\xi )} }{\max (1,\mp\operatorname{sgn}(\Re \xi )\left| \Re \xi \right|^\rho)}\sum\limits_{n = N}^\infty  \frac{1}{(n + 1)^{\Re u/\omega } } 
\\ & \le \frac{C^ \pm}{2^{\sigma /\omega } }\frac{\Gamma (\Re u/\omega )}{\omega  - \sigma }\frac{\e^{V^{\pm}(\sigma ,\xi )} }{\max (1,\mp\operatorname{sgn}(\Re \xi )\left| \Re \xi \right|^\rho)}\sum\limits_{n = N}^\infty \int_{n + 1/2}^{n + 3/2} \frac{\d t}{t^{\Re u/\omega }}
\\ & = \frac{C^ \pm}{2^{\sigma /\omega } }\frac{\Gamma (\Re u/\omega  - 1)}{\omega  - \sigma }\frac{\e^{V^{\pm}(\sigma ,\xi )} }{\max (1,\mp\operatorname{sgn}(\Re \xi )\left| \Re \xi \right|^\rho)}\left(N+\frac{1}{2}\right)^{1 - \Re u/\omega } 
\end{align*}
for $\Re u>\omega >\sigma>0$, $\xi \in \Gamma^\pm(d)$ and $N\geq 0$. In the second step use has been made of the known inequality for the ratio of two gamma functions \cite[\href{http://dlmf.nist.gov/5.6.E8}{Eq. 5.6.8}]{DLMF}, and in the third step we employed the Hermite--Hadamard inequality for convex functions \cite[Ch. 1, Eq. (1.32)]{Niculescu2018}.

\section{Applications}\label{Section6}

In this section, we give two illustrative examples to demonstrate the applicability of our theory.

\subsection{A radial Schr\"odinger equation}
As a first application of the theory, we consider the radial Schr\"odinger equation
\begin{equation}\label{radialeq}
\frac{\d^2 w(u,z)}{\d z^2 } = u^2 \left( \frac{(z - 1)^2}{4} - \left(\lambda +\tfrac{1}{2}\right)\frac{1}{u} + \frac{\ell (\ell  + 1)}{z^2}\frac{1}{u^2} \right)w(u,z),
\end{equation}
which is associated with the rotating harmonic oscillator. In physical applications, the parameters $u$, $\lambda$ and $\ell$ are real and non-negative, with $\ell$ an integer, and $u$ large (see, e.g., \cite{Froman1978}). We introduce a branch cut along the real axis from $1$ to $-\infty$ so that $z$ is restricted to $\mathbf{D} = \left\{ z:\left| \arg (z - 1) \right| < \pi \right\}$. Application of the transformation
\[
\xi  = \xi (z) = \int_1^z \frac{t - 1}{2}\d t  = \frac{(z - 1)^2}{4}, \quad w(u,z) = \left(\frac{2}{z - 1}\right)^{1/2} W(u,\xi ),
\]
brings \eqref{radialeq} into the standard form \eqref{Eq}, where
\begin{equation}\label{radialeq2}
\phi (\xi ) =  - \left(\lambda +\tfrac{1}{2}\right)\frac{1}{\xi} \quad \text{and} \quad \psi (\xi ) = \frac{\ell (\ell  + 1)}{z^2 \xi } - \frac{3}{16\xi ^2} .
\end{equation}
The domain $\mathbf{D}$ is mapped into the sectorial region $\mathbf{G} = \left\{ \xi :\left| \arg \xi \right| < 2\pi \right\}$ on the universal covering of $\mathbb{C} \setminus \left\{0\right\}$.

To apply Theorem \ref{thm1}, we have to ensure that either Condition \ref{cond1} or Condition \ref{condv2} is satisfied. First, we have to define a suitable subdomain $\Delta$ of $\mathbf{G}$. To this end, observe that $\phi(\xi)$ becomes unbounded as $\left| \xi  \right| \to 0$ ($z\to 1$). Similarly, $\psi(\xi)$ is unbounded when $\left| \xi  \right| \to 0$ or $\xi  \to \frac{1}{4}\e^{ \pm 2\pi \im}$ ($z\to 1$ or $|z|\to 0$). Hence, it is natural to fix $0<\eps \ll \frac{1}{4}$ and define
\[
\Delta  = \left\{ \xi :\xi  \in \mathbf{G},\, \left| \xi  \right| > \eps,\, \left| \xi  - \tfrac{1}{4}\e^{ \pm 2\pi \im} \right| > \eps  \right\}.
\]
The functions $\phi(\xi)$ and $\psi(\xi)$ are analytic and bounded throughout $\Delta$ and thus, since $\xi \sim \frac{1}{4}z^2$ for large values of $|z|$, Condition \ref{condv2} is satisfied with $\rho = 1/2$. Let $d>0$ be arbitrary and define the subdomains $\Gamma^\pm (d)$ of $\Delta$ by
\[
\Gamma^+(d) = \left\{ \xi :0 < \arg \xi  < 2\pi ,\, \left| \xi  \right| > d + \eps  \right\}\setminus \left\{ \xi :\Re \xi  \ge 0,\, \left| \Im \xi \right| \le d + \eps  \right\}
\]
and
\[
\Gamma^-(d) = \left\{ \xi :\left| \arg \xi \right| < \pi ,\, \left| \xi \right| > d + \eps  \right\}\setminus \left\{ \xi :\Re \xi  \le 0,\, \left| \Im \xi \right| \le d + \eps \right\}.
\]
These domains $\Gamma^+(d)$ and $\Gamma^-(d)$ and the corresponding $z$-domains (which we label $D^+(d)$ and $D^-(d)$), are depicted in Figures \ref{Figure1} and \ref{Figure2}. Taking $\alpha^\pm = 1$ in \eqref{WKBsol} and employing Theorem \ref{thm1}, we find that the differential equation \eqref{Eq}, with $\phi(\xi)$ and $\psi(\xi)$ specified in \eqref{radialeq2}, has solutions of the form
\[
W^ \pm  (u,\xi ) = \xi ^{ \mp (2\lambda+1) /4} \e^{ \pm u\xi } \left( 1 + \eta ^ \pm  (u,\xi ) \right),
\]
which are analytic in $\left\{ u:\Re u > 0 \right\} \times \Gamma^\pm (d)$ and satisfy the limit conditions \eqref{limitreq}. Returning to the original variables of \eqref{radialeq}, we write $\mu^\pm (u,z) = \eta^\pm (u,\xi (z))$, and
\[
w^\pm (u,z) =\left(\frac{2}{z - 1}\right)^{1/2} W^ \pm (u,\xi ) = \left( \frac{z - 1}{2} \right)^{ \mp (\lambda  + 1/2) - 1/2} \exp\!\left(\pm u\frac{(z - 1)^2}{4}\right) \left( 1 + \mu^\pm (u,z) \right).
\]
These solutions are holomorphic in $\left\{ u:\Re u > 0 \right\} \times D^\pm (d)$, and $\mu^\pm (u,z) \to 0$ as $z \to \im \infty$\footnote{Throughout this section we use the convention that $\pm \im = \e^{ \pm \frac{\pi}{2}\im}$.} and $z \to  + \infty$, respectively.

Asymptotic and convergent series for $w^\pm (u,z)$ are given by Theorems \ref{thm1}--\ref{thm3}. In terms of $z$, the asymptotic expansion coefficients $\mathsf{A}_n^\pm = \mathsf{A}_n^\pm(\xi(z))$ extend analytically to the whole of $\mathbf{D}$ and are defined recursively by $\mathsf{A}_0^\pm =1$ and
\begin{gather}\label{ex1rec}
\begin{split}
& \mathsf{A}_{n + 1}^ \pm = \left( \lambda  + \tfrac{1}{2} \right)\frac{2}{(z - 1)^2}\mathsf{A}_n^\pm \mp \frac{1}{z - 1}\frac{\d \mathsf{A}_n^ \pm }{\d z} \\ & \mp \left( \frac{3}{4}  \mp 2\left( \lambda  + \tfrac{1}{2} \right) + \left( \lambda  + \tfrac{1}{2} \right)^2 \right)\int_{\mathscr{P}( + \infty ,z)} \frac{\mathsf{A}_n^ \pm}{(t - 1)^3}\d t  \pm \ell (\ell  + 1)\int_{\mathscr{P}( + \infty ,z)} \frac{\mathsf{A}_n^ \pm}{t^2 (t - 1)}\d t 
\end{split}
\end{gather}
for $n\geq 0$. In particular,
\[
\mathsf{A}_1^ \pm   =  \pm \left( \frac{3}{8} \pm \left( \lambda  + \tfrac{1}{2} \right) + \frac{1}{2}\left( \lambda  + \tfrac{1}{2} \right)^2 \right)\frac{1}{(z - 1)^2 } \pm \ell (\ell  + 1)\left( \log \! \left( 1 - \frac{1}{z} \right) + \frac{1}{z} \right) .
\]
The calculation of the higher coefficients quickly becomes rather laborious, and hence the alternative method discussed in Appendix \ref{Appendix} may be preferred over \eqref{ex1rec}. The coefficients $\mathsf{B}_n^\pm$ of the corresponding factorial series expansions can be computed via the Stirling algorithm \eqref{Balt}.

\begin{figure}[ht]
    \begin{subfigure}[c]{0.45\textwidth}
			\centering
			\includegraphics[width=\textwidth]{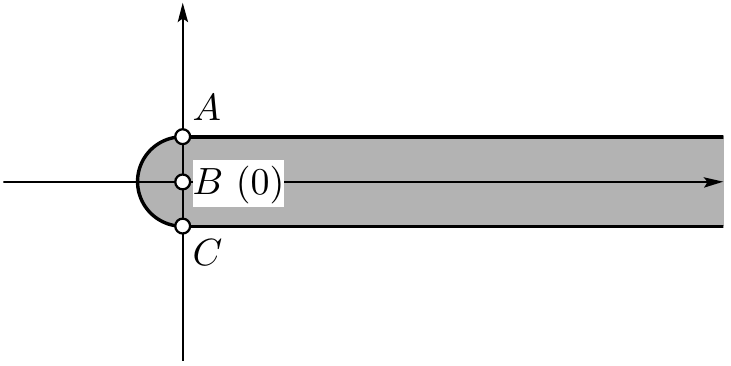}
      \caption{}
    \end{subfigure}
    \hfill
    \begin{subfigure}[c]{0.5\textwidth}
		\centering 
      \includegraphics[width=\textwidth]{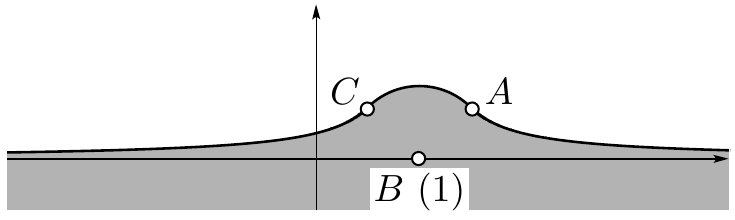}
      \caption{}
    \end{subfigure}
    \caption{(a) Domain $\Gamma^+(d)$ (unshaded). (b) Domain $D^+(d)$ (unshaded).}
		\label{Figure1}
  \end{figure}
	
	\begin{figure}[ht]
    \begin{subfigure}[c]{0.45\textwidth}
			\centering
			\includegraphics[width=\textwidth]{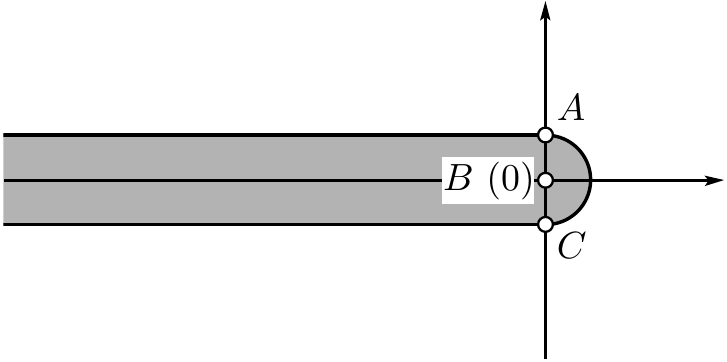}
      \caption{}
    \end{subfigure}
    \hspace{20pt}
    \begin{subfigure}[c]{0.225\textwidth}
		\centering 
      \includegraphics[width=\textwidth]{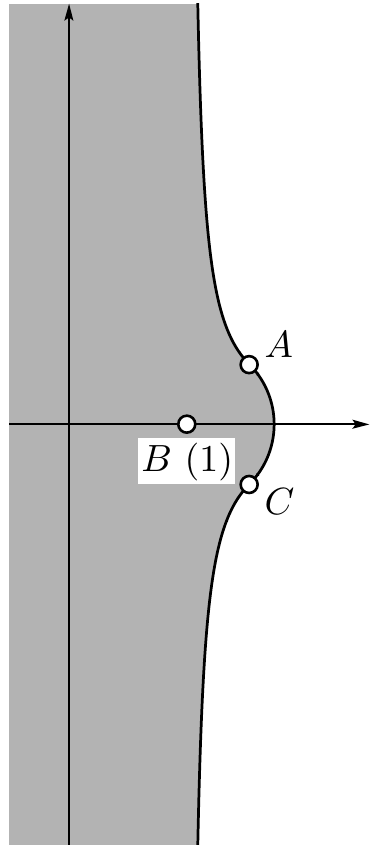}
      \caption{}
    \end{subfigure}
    \caption{(a) Domain $\Gamma^-(d)$ (unshaded). (b) Domain $D^-(d)$ (unshaded).}
		\label{Figure2}
  \end{figure}

We remark that Dunster \cite{Dunster1998} considered the problem of finding rigorous asymptotic expansions of the eigenvalues of the rotating harmonic oscillator. His analysis relies on WKB theoretic methods applied to the equation \eqref{radialeq}, although his definition of the WKB solutions is different from ours.

\subsection{The Bessel equation}
As a second application, we consider the Bessel equation
\[
\frac{\d^2 w(\nu ,z)}{\d z^2 } = \nu ^2 \left( \frac{1}{z^2} + \frac{2\kappa}{z^2}\frac{1}{\nu} + \left( \frac{4\kappa ^2  - 1}{4z^2 } - 1 \right)\frac{1}{\nu ^2 }\right)w(\nu ,z),
\] 
where $w(\nu ,z) = z^{1/2} H_{\nu  + \kappa }^{(1)} (z)$, $z^{1/2} H_{\nu  + \kappa }^{(2)} (z)$, or any linear combination of these functions \cite[\href{http://dlmf.nist.gov/10.13.E1}{Eq. 10.13.1}]{DLMF}. The parameter $\kappa$ can assume arbitrary complex values. With $\nu$ playing the role of the parameter $u$, the transformations \eqref{Liouville} could be applied with $f_0(z)=z^{-2}$. Unfortunately, the functions $\phi (\xi)$ and $\psi(\xi)$ appearing in the resulting equation satisfy neither of the requirements posed in Condition \ref{cond1} or \ref{condv2}. To overcome this obstruction we replace $z$ by $\nu z$, and consider instead the equation
\begin{equation}\label{Besseleq}
\frac{\d^2 w(\nu ,z)}{\d z^2} = \nu ^2 \left( \frac{1 - z^2}{z^2} + \frac{2\kappa}{z^2}\frac{1}{\nu} + \frac{4\kappa ^2  - 1}{4z^2 }\frac{1}{\nu ^2 } \right)w(\nu ,z),
\end{equation}
which has particular solutions $w(\nu ,z) = z^{1/2} H_{\nu  + \kappa }^{(1)} (\nu z)$ and $w(\nu ,z) = z^{1/2} H_{\nu  + \kappa }^{(2)} (\nu z)$. This equation can now be treated by means Theorem \ref{thm1}. Similarly to the first example, we make a branch cut along the real axis from $1$ to $-\infty$ and restrict $z$ to the domain $\mathbf{D} = \left\{ z:\left| \arg (z - 1) \right| < \pi \right\}$. The transformation
\begin{equation}\label{Liouville3}
\xi  = \xi (z) = \int_1^z \frac{(1-t^2)^{1/2}}{t} \d t  = \im((z^2  - 1)^{1/2}  - \arcsec z ),\quad w(\nu ,z) = -\im\frac{z^{1/2}}{(z^2  - 1)^{1/4}} W(\nu ,\xi)
\end{equation}
brings \eqref{Besseleq} into the standard form \eqref{Eq} (with $\nu$ in place of $u$), with
\begin{equation}\label{Besseleq2}
\phi (\xi ) = -\frac{2\kappa }{z^2-1 } \quad \text{and} \quad \psi (\xi ) = \frac{1}{4}\frac{1-4\kappa ^2 }{z^2  - 1} + \frac{3}{2}\frac{1}{(z^2  - 1)^2} + \frac{5}{4}\frac{1}{(z^2  - 1)^3}.
\end{equation}
The image of $\mathbf{D}$ under the mapping \eqref{Liouville3} can be determined by the following considerations:
\begin{enumerate}[(i)]
	\item When $z>1$, $\xi$ is purely imaginary; $z=1$, $+\infty$ corresponding to $\xi=0$, $\im\infty$, respectively.
	\item For $z$ near $1$, $\xi\sim\frac{2\sqrt 2}{3}\im(z - 1)^{3/2}$; for large $|z|$, $\xi \sim \im z$.
	\item If $z=1+x\e^{\pm\pi \im}$ ($0<x<1$), then $\pm\xi$ is real and positive; $z=1+\e^{\pm\pi \im}$ corresponding to $\xi = \infty\e^{2\pi \im}$, $\infty\e^{-\pi \im}$.
	\item If $z=1+x\e^{\pm\pi \im}$ ($1<x<2$), then $\xi+\pi \im$ is real and negative; $z=1+2\e^{\pm\pi \im}$ corresponding to $\xi = \pi \e^{\frac{3\pi}{2}\im}$, $-\pi \im$.
	\item If $z= 1+x\e^{\pm\pi \im}$ ($x>2$), then $\xi$ is purely imaginary; $z=1+\infty\e^{\pm\pi \im}$ corresponding to $\xi = \infty \e^{\frac{3\pi}{2} \im}$, $-\im\infty$.
\end{enumerate}
Thus $\mathbf{D}$ is transformed into the following subdomain of the universal covering of $\mathbb{C}\setminus \left\{0\right\}$:
\[
\mathbf{G} = \left\{ \xi : - \tfrac{\pi}{2} < \arg \xi  < \tfrac{3\pi}{2} \right\} \cup \left\{ \xi : -\pi < \arg \xi \leq - \tfrac{\pi}{2} \; \text{or} \; \tfrac{3\pi}{2} \leq \arg \xi  < 2\pi ,\, - \pi < \Im \xi < 0 \right\} .
\]
To employ Theorem \ref{thm1}, we have to choose an appropriate subdomain $\Delta$ of $\mathbf{G}$ for which either Condition \ref{cond1} or Condition \ref{condv2} holds. We note that $\phi (\xi )$ and $\psi(\xi)$ become singular as $z \to  \pm 1$ or, equivalently, as $\xi  \to 0$ or $\xi \to -\pi \im$, $\pi \e^{\frac{3\pi}{2}\im}$. Thus, it is natural to fix $0<\eps \ll \pi$ and define
\[
\Delta  = \left\{ \xi :\xi  \in \mathbf{G},\, \left| \xi  \right| > \eps,\, \left| \xi + \pi \im \right| > \eps,\, \left| \xi - \pi \e^{\frac{3\pi}{2}\im} \right| > \eps \right\} .
\]
Taking into account that $\xi \sim \im z$ for large values of $|z|$ and that $\phi (\xi )$, $\psi(\xi)$ are bounded analytic functions on $\Delta$, we can assert that Condition \ref{cond1} holds with $\rho = 1$. Let $d>0$ and define the subdomains $\Gamma^\pm (d)$ of $\Delta$ by
\begin{multline*}
\Gamma^+ (d) = \left\{ \xi :0 < \arg \xi  < 2\pi ,\,\left| \xi  \right| > d + \eps ,\,\left| \xi  - \pi \e^{\frac{3\pi}{2}\im} \right| > d + \eps \right\}\\
\setminus \left( \left\{ \xi :\Re \xi  \ge 0,\,\left| \Im \xi \right| \le d + \eps \; \text{or} \; \Im \xi  \le  - \pi  + d + \eps \right\} \cup \left\{ \xi :\Re \xi  \ge -d ,\,\Im \xi  <  - \pi \right\} \right)
\end{multline*}
and
\begin{multline*}
\Gamma^- (d) = \left\{ \xi :\left| \arg \xi \right| < \pi ,\, \left| \xi  \right| > d + \eps ,\, \left| \xi + \pi \im \right| > d + \eps \right\} \\
\setminus \left( \left\{ \xi :\Re \xi \leq 0,\,\left| \Im \xi \right| \le d + \eps \; \text{or} \; \Im \xi  \le  - \pi  + d + \eps \right\} \cup \left\{ \xi :\Re \xi \leq d ,\,\Im \xi  <  - \pi \right\} \right) ,
\end{multline*}
respectively. The domains $\Gamma^\pm(d)$ and the corresponding $z$-domains $D^\pm(d)$ are shown in Figures \ref{Figure3} and \ref{Figure4}. Choosing $\alpha^\pm =\im\infty$ in \eqref{WKBsol} and applying Theorem \ref{thm1}, we find that the differential equation \eqref{Eq}, with $\phi(\xi)$ and $\psi(\xi)$ specified in \eqref{Besseleq2}, has solutions of the form
\[
W^ \pm  (\nu,\xi ) = \left( \frac{\xi  - 1}{\xi  + 1} \right)^{ \pm \kappa /2} \e^{ \pm \nu\xi } \left( 1 + \eta ^ \pm  (\nu,\xi ) \right),
\]
which are analytic in $\left\{ \nu:\Re \nu > 0 \right\} \times \Gamma^\pm (d)$ and satisfy \eqref{limitreq}. Returning to the original variables of \eqref{Besseleq}, we write $\mu^\pm (\nu,z) = \eta^\pm (\nu,\xi (z))$, and
\[
w^ \pm  (\nu ,z) =  -\im\frac{z^{1/2}}{(z^2  - 1)^{1/4}}W^\pm (\nu ,\xi ) =  -\im\left( \frac{\xi  - 1}{\xi  + 1} \right)^{ \pm \kappa /2} \frac{z^{1/2} \e^{ \pm \nu \xi } }{(z^2  - 1)^{1/4}}\left( 1 + \mu^\pm (\nu ,z) \right) .
\]
These solutions are holomorphic in $\left\{ \nu:\Re \nu > 0 \right\} \times D^\pm (d)$, and $\mu^\pm (\nu,z) \to 0$ as $z \to \pm\im \infty$. 

To identify the standard solutions $z^{1/2} H_{\nu  + \kappa }^{(1)} (\nu z)$ and $z^{1/2} H_{\nu  + \kappa }^{(2)} (\nu z)$ of \eqref{Besseleq} in terms of $w^+ (\nu,z)$ and $w^- (\nu,z)$, observe that for fixed $\nu$, $w^+ (\nu,z)$ is recessive as $\xi\to \infty \e^{\pi \im}$, that is, as $z \to \im\infty$, and $w^- (\nu,z)$ is recessive as $\xi\to +\infty$, that is, as $z \to -\im\infty$. Consequently, $z^{1/2} H_{\nu  + \kappa }^{(1)} (\nu z)$ is a multiple\footnote{By ``multiple'' we mean a number that is independent of $z$, but may depend on $\nu$ and $\kappa$.} of $w^+ (\nu,z)$ and $z^{1/2} H_{\nu  + \kappa }^{(2)} (\nu z)$ is a multiple of $w^- (\nu,z)$. A convenient way of identifying these multiples is to compare the limiting forms of the solutions as $|z|\to \infty$. From \cite[\href{http://dlmf.nist.gov/10.17.i}{\S10.17(i)}]{DLMF} and $\xi  = \im z - \frac{\pi }{2}\im + \O(\left| z \right|^{ - 1} )$, we have
\begin{align*}
z^{1/2} H_{\nu  + \kappa }^{(1,2)} (\nu z)  & = \left(\frac{2}{\pi \nu}\right)^{1/2} \exp \! \left(  \pm \im\left( \nu z - \frac{\pi}{2}\nu  - \frac{\pi }{4} \right) \right) \left( 1 + \O\!\left(\frac{1}{|z|}\right) \right) \\ & = 
\e^{ \mp \frac{\pi}{4} \im} \left( \frac{2}{\pi \nu } \right)^{1/2} \left( \frac{\xi  - 1}{\xi  + 1} \right)^{ \pm \kappa /2} \frac{z^{1/2} \e^{ \pm \nu \xi } }{(z^2  - 1)^{1/4}}\left( 1 + \O\!\left( \frac{1}{\left| z \right|} \right) \right)
\end{align*}
as $z \to \pm\im\infty$ with $\nu$, $\left| \arg \nu \right| < \frac{\pi}{2}$, and $\kappa$ being fixed. Thus we deduce that for $\Re \nu >0$ and $z \in D^\pm (d)$,
\begin{equation}\label{Hankel}
H_{\nu  + \kappa }^{(1,2)} (\nu z) = \e^{\left( \frac{\pi}{2} \mp \frac{\pi}{4} \right)\im} \left( \frac{2}{\pi \nu z} \right)^{1/2} w^\pm (\nu ,z ).
\end{equation}
Finally, using the well-known relationship between the Hankel and Bessel functions \cite[\href{http://dlmf.nist.gov/10.4.E4}{Eq. 10.4.4}]{DLMF}, we find from \eqref{Hankel} that
\begin{align}
J_{\nu  + \kappa } (\nu z) & = \frac{1}{(2\pi \nu z)^{1/2}}\left( \e^{\frac{\pi}{4}\im} w^+ (\nu ,z) - \e^{-\frac{\pi}{4}\im} w^- (\nu ,z) \right),\\ Y_{\nu  + \kappa } (\nu z) & = \frac{1}{(2\pi \nu z)^{1/2}}\left( \e^{-\frac{\pi}{4}\im} w^+ (\nu ,z) - \e^{\frac{\pi}{4}\im} w^- (\nu ,z) \right),\label{Bessel}
\end{align}
provided $\Re \nu >0$ and $z \in D^+ (d) \cap D^- (d)$. For Borel re-summed WKB expansions of the Bessel function $J_{\nu  + \kappa } (\nu z)$ in other $z$-regions, the interested reader is referred to \cite{Takahashi2019}.

\begin{figure}[ht]
    \begin{subfigure}[c]{0.45\textwidth}
			\centering
			\includegraphics[width=\textwidth]{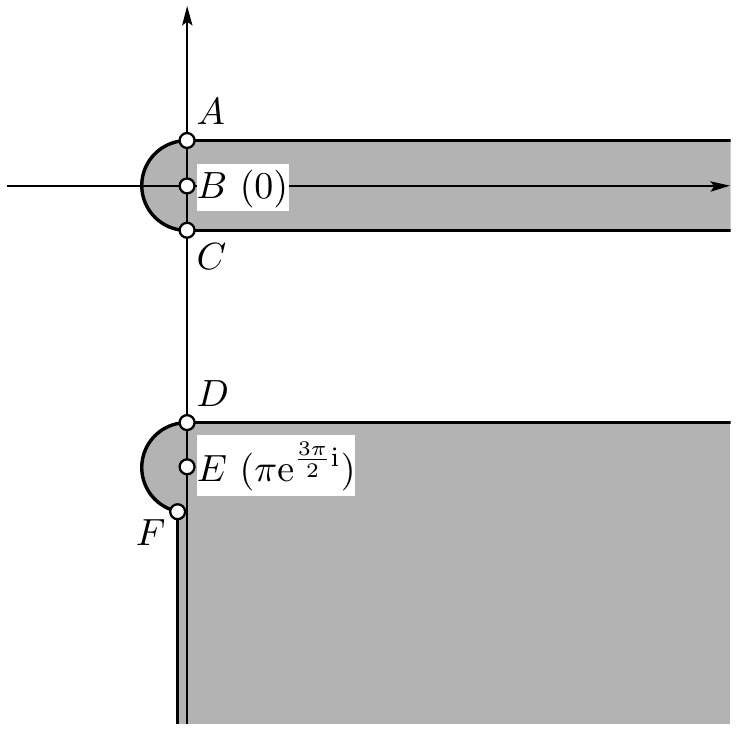}
      \caption{}
    \end{subfigure}
    \hfill
    \begin{subfigure}[c]{0.45\textwidth}
		\centering 
      \includegraphics[width=\textwidth]{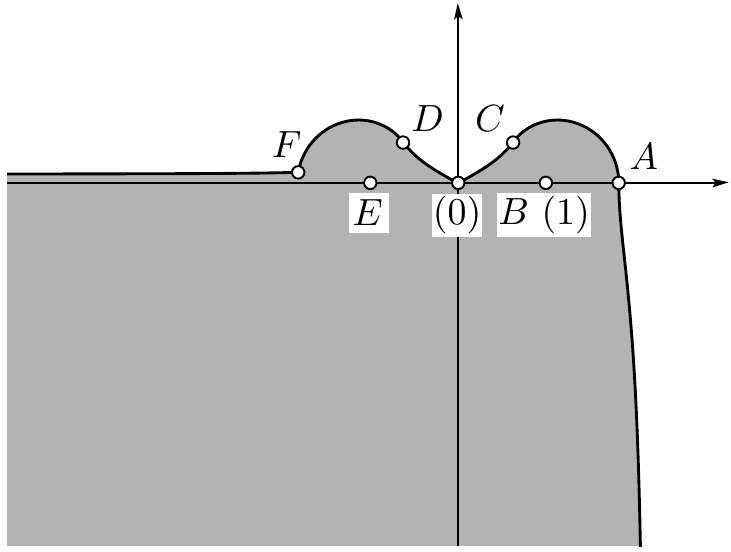}
      \caption{}
    \end{subfigure}
    \caption{(a) Domain $\Gamma^+(d)$ (unshaded). (b) Domain $D^+(d)$ (unshaded).}
		\label{Figure3}
  \end{figure}

\begin{figure}[ht]
    \begin{subfigure}[c]{0.45\textwidth}
			\centering
			\includegraphics[width=\textwidth]{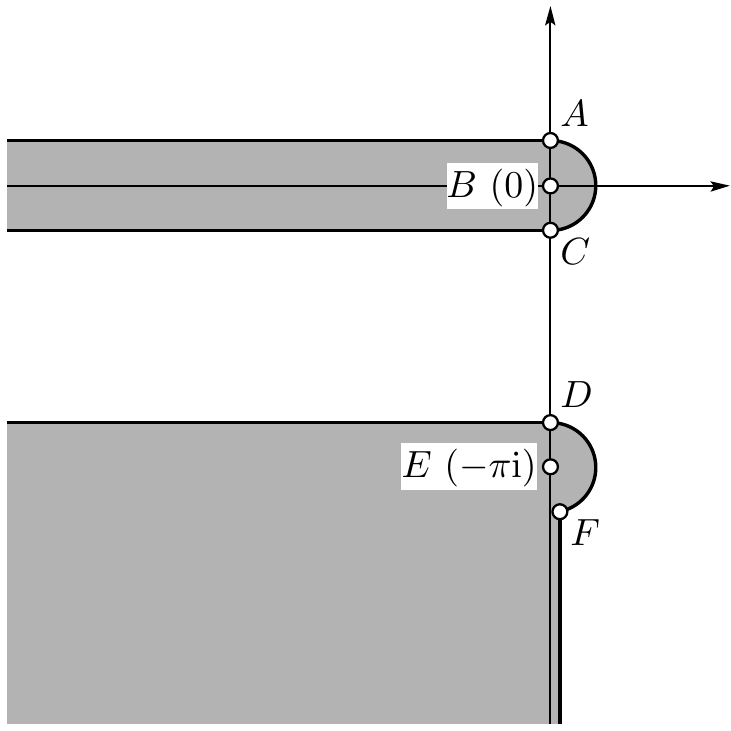}
      \caption{}
    \end{subfigure}
    \hfill
    \begin{subfigure}[c]{0.45\textwidth}
		\centering
      \includegraphics[width=\textwidth]{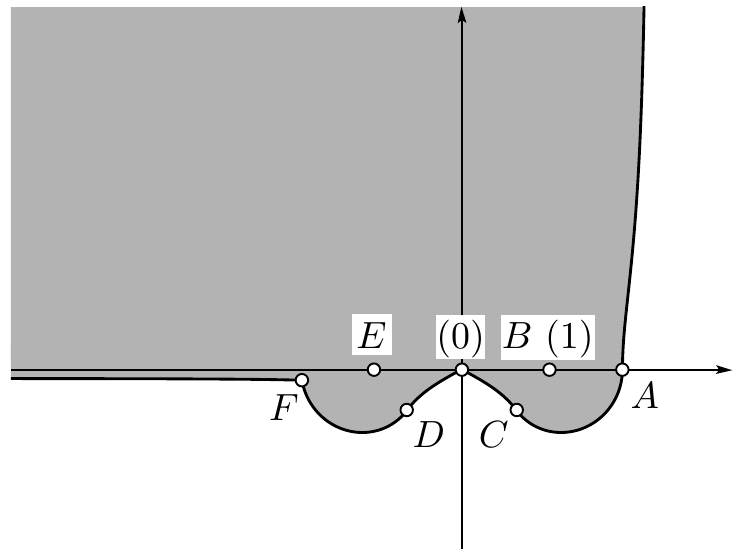}
      \caption{}
    \end{subfigure}
    \caption{(a) Domain $\Gamma^-(d)$ (unshaded). (b) Domain $D^-(d)$ (unshaded).}
		\label{Figure4}
  \end{figure}

Asymptotic and convergent series for $w^\pm (\nu,z)$ follow from Theorems \ref{thm1}--\ref{thm3}. In the special case that $z>1$ and $\kappa=0$, the results \eqref{Hankel}--\eqref{Bessel} combined with the asymptotic expansions of $w^\pm (\nu,z)$ reproduce the well-known Debye expansions (see, e.g., \cite[Ch. II, \S4]{Wong2001}). For positive real $\nu$ and $z$, and for real $\kappa$, the functions $H_{\nu + \kappa}^{(1)} (\nu z)$ and $H_{\nu + \kappa}^{(2)} (\nu z)$ are complex conjugates of each other. Accordingly,
\[
\mathsf{A}_n^- (\xi ) = \mathsf{A}_n^+ (\xi \e^{\pi \im} ),
\]
hence we can omit the superscript $+$ and write $\mathsf{A}_n^+ = \mathsf{A}_n$. In terms of $z$, the asymptotic expansion coefficients $\mathsf{A}_n = \mathsf{A}_n (\xi(z))$ continue analytically to the whole of $\mathbf{D}$ and are given recursively by $\mathsf{A}_0 =1$ and
\begin{align*}
\mathsf{A}_{n + 1}  = \frac{\kappa }{z^2  - 1}\mathsf{A}_n & + \frac{\im}{2}\frac{z}{(z^2  - 1)^{1/2}}\frac{\d \mathsf{A}_n }{\d z} \\ & + \frac{\im}{8}\int_{\mathscr{P}( + \infty ,z)} t\left( \frac{1-4\kappa ^2}{(t^2  - 1)^{3/2}} - \frac{8\im\kappa}{(t^2  - 1)^2}+ \frac{5}{(t^2  - 1)^{5/2}} \right)\mathsf{A}_n \d t
\end{align*}
for $n\geq 0$. In particular,
\[
\mathsf{A}_1 = - \frac{\im}{8}\frac{1-4\kappa ^2}{(z^2  - 1)^{1/2}} + \frac{1}{2}\frac{\kappa }{z^2  - 1} - \frac{5\im}{24}\frac{1}{(z^2  - 1)^{3/2}}.
\]
It is expected from this result that higher coefficients may be expressed conveniently in terms of the variable
\[
p = -\im(z^2  - 1)^{-1/2}.
\]
We find in terms of $p$
\[
\mathsf{A}_{n + 1} = - \kappa p^2 \mathsf{A}_n + \frac{1}{2}p^2 (1 - p^2)\frac{\d \mathsf{A}_n}{\d p} + \frac{1}{8}\int_0^p (1- 4\kappa ^2 + 8\kappa t -5t^2)\mathsf{A}_n \d t .
\]
It is easily deduced from this relation that $\mathsf{A}_n$ is a polynomial in $p$ of degree $3n$. The first four polynomials are found to be
\begin{gather*}
\mathsf{A}_0 = 1, \quad \mathsf{A}_1 = \frac{1-4\kappa ^2}{8}p - \frac{\kappa}{2}p^2 - \frac{5}{24}p^3,
\\ \mathsf{A}_2 = \frac{9 - 40\kappa ^2  + 16\kappa ^4 }{128}p^2 - \frac{29\kappa  - 20\kappa ^3 }{48}p^3  - \frac{77 - 140\kappa ^2 }{192}p^4 + \frac{35\kappa }{48}p^5  + \frac{385}{1152}p^6, \\
\mathsf{A}_3 = \frac{225 - 1036\kappa ^2  + 560\kappa ^4  - 64\kappa ^6 }{3072}p^3  - \frac{751\kappa  - 728\kappa ^3  + 112\kappa ^5 }{768}p^4 \\ - \frac{4563 - 12040\kappa ^2  + 2800\kappa ^4 }{5120}p^5  + \frac{3619\kappa  - 1540\kappa ^3 }{1152}p^6  + \frac{17017 - 20020\kappa ^2 }{9216}p^7  - \frac{5005}{2304}\kappa p^8  - \frac{85085}{82944}p^9 .
\end{gather*}
The coefficients of the corresponding factorial series expansions can be computed by applying similar transformations to \eqref{Brec} or by the Stirling algorithm \eqref{Balt}.

\section{Discussion}\label{Section7}

We studied the Borel summability of WKB solutions of certain Schr\"odinger-type differential equations with a large parameter. It was shown that under mild requirements on the potential function of the equation, the WKB solutions are Borel summable with respect to the large parameter in vast, unbounded domains of the independent variable. We demonstrated that the formal WKB expansions are the asymptotic expansions, uniform with respect to the independent variable, of the Borel re-summed solutions and supplied computable bounds on their remainder terms. In addition, it was proved that the WKB solutions can be expressed using factorial series in the parameter, and that these expansions converge in half-planes, uniformly in the independent variable.

The theory presented here is a first step towards the global analysis of the WKB solutions of the differential equation \eqref{Eq0}. The main result of the paper demonstrates that the WKB solutions are Borel summable provided that the independent variable is bounded away from Stokes curves emerging either from transition points of $f_0(z)$ or from singular points of $f_1(z)$ and $f_2(z)$. In particular, it does not give any information regarding connection formulae joining the WKB solutions at either side of the Stokes curves.

In what follows, we discuss briefly a possible extension of the theory to turning point problems. Turning points are the simplest types of transition points. Consider the differential equation \eqref{Eq0}, and suppose that the functions $f_n(z)$ are analytic in a domain $\mathbf{D}$ and $f_0(z)$ vanishes at exactly one point $z_0$, say, of $\mathbf{D}$. For the sake of simplicity, we assume that $z_0$ is a simple zero of $f_0(z)$, i.e., $z_0$ is a simple turning point of \eqref{Eq0}. Following Olver \cite[Ch. 11, \S11]{Olver1997}, we make the transformations
\begin{equation}\label{transform2}
\zeta  = \zeta (z) = \left( \frac{3}{2}\int_{z_0 }^z f_0^{1/2} (t)\d t \right)^{2/3} ,\quad f_0^{1/4} (z)w(u,z) = \zeta ^{1/4} \mathcal{W}(u,\zeta ) .
\end{equation}
Then \eqref{Eq0} becomes
\begin{equation}\label{Eq2}
\frac{\d^2 \mathcal{W}(u,\zeta )}{\d\zeta ^2 } = \left( u^2 \zeta  + u\Phi (\zeta ) + \Psi (\zeta ) \right)\mathcal{W}(u,\zeta ),
\end{equation}
with
\[
\Phi (\zeta ) = \frac{\zeta f_1 (z)}{f_0 (z)} \quad \text{and} \quad \Psi (\zeta ) = \frac{\zeta f_2 (z)}{f_0 (z)} - \frac{\zeta ^{3/4} }{f_0^{3/4} (z)}\frac{\d^2 }{\d z^2 }\frac{\zeta ^{1/4} }{f_0^{1/4} (z)} .
\]
The functions $\Phi (\zeta )$ and $\Psi (\zeta )$ are holomorphic in the corresponding $\zeta$ domain $\mathbf{H}$, say. The transformation \eqref{transform2} maps the simple turning point $z_0$ into the origin in $\mathbf{H}$, and the three Stokes curves issuing from $z_0$ are mapped into the rays $\arg \zeta  = 0, \pm \frac{2\pi}{3}\im$. The works of Boyd \cite{Boyd1987} and Olver \cite[Ch. 11, \S11]{Olver1997}  on turning point problems suggest that, under suitable assumptions on $\Phi (\zeta )$ and $\Psi (\zeta )$, the set of all solutions of the differential equation \eqref{Eq2} is of the form
\begin{equation}\label{uniformsol}
\begin{aligned}
\mathcal{W}(u,\zeta ) = \Ai \!\left( u^{2/3} \zeta  + \frac{1}{2u^{1/3} \zeta ^{1/2}}\int_{0}^\zeta  \frac{\Phi (t)}{t^{1/2}}\d t  \right)& \mathscr{A}(u,\zeta ) \\ + \frac{1}{u^{4/3}}{\Ai}\,'\!\! & \left( u^{2/3} \zeta  + \frac{1}{{2u^{1/3} \zeta ^{1/2} }}\int_{0}^\zeta  \frac{\Phi (t)}{t^{1/2} }\d t \right)\mathscr{B}(u,\zeta ),
\end{aligned}
\end{equation}
where $\Ai$ denotes any solution of Airy's equation. We expect that, for sufficiently large values of $|u|$, the coefficient functions $\mathscr{A}(u,\zeta )$ and $\mathscr{B}(u,\zeta )$ are holomorphic functions of $\zeta$ in an appropriate subdomain of $\mathbf{H}$ including $\zeta=0$ and the Stokes rays $\arg \zeta  = 0, \pm \frac{2\pi}{3}\im$, and as functions of $u$ are described asymptotically by descending power series in $u$. Consequently, the continuation formulae for solutions of the form \eqref{uniformsol} follow directly from those for the Airy functions. Then, with the aid of the transformations
\[
\xi = \tfrac{2}{3}\zeta^{3/2} ,\quad W(u,\xi ) = \zeta ^{1/4} \mathcal{W}(u,\zeta ),
\]
connection formulae can be established for formal solutions of the form \eqref{WKB1}. It is also possible to use the foregoing analysis to extend the Borel summability results of the WKB solutions \eqref{WKBsol} to the vicinity of Stokes rays emanating from $\xi=0$. First, we identify the solutions of the form \eqref{uniformsol} which correspond to $W^\pm(u,\xi)$. Next, by employing Theorem \ref{thm1} and techniques similar to those in \cite{Dunster2001}, it is shown that the asymptotic expansions of the coefficient functions $\mathscr{A}(u,\zeta )$ and $\mathscr{B}(u,\zeta )$ are Borel summable with respect to the large parameter $u$, and that the corresponding Borel transforms are analytic functions of $\zeta$ in an unbounded domain containing the Stokes rays. The process is completed by an appeal to a theorem on the composition of Borel summable series \cite[Ch. 5, Theorem 5.55]{Mitschi2016} and the well-known Borel summability properties of WKB solutions of the Airy equation.

In Section \ref{Section3}, we showed that the Borel transforms $F^\pm(t,\xi)$ can be continued analytically, in the complex variable $t$, to a strip containing the positive real axis. The method, however, does not provide us with any information regarding the nature and location of the singularities of $F^\pm(t,\xi)$ in the $t$-plane. Explicit knowledge of singularities is an important prerequisite to any investigation in resurgent analysis and exponential asymptotics. By looking at the functional equation \eqref{finalinteq}, for example, it is reasonable to expect that the singular points of $F^-(t,\xi)$ are located at
\[
t = - 2\xi  + 2\xi _k,
\] 
where the $\xi _k$'s, $k \in \mathbb{Z}_{\geq 0}$, are the singularities of the functions $\phi(\xi)$ and $\psi(\xi)$. Hence, there seems to be a direct link between the singular points of $F^-(t,\xi)$ in the $t$-plane and those of $\phi(\xi)$ and $\psi(\xi)$ in the $\xi$-plane. This conjecture may be verified for some specific examples where the Borel transforms are explicitly known (e.g., the Airy or the Legendre equation), but the general case requires further investigation.

Finally, it would be of great interest to extend the results of the paper to differential equations of the type \eqref{Eq0} in which the potential function $f(u,z)$ admits a (Borel summable) uniform asymptotic expansion of the form
\[
f(u,z) \sim \sum\limits_{n = 0}^\infty \frac{f_n (z)}{u^n}
\]
as $|u|\to +\infty$. Asymptotic properties of solutions to such equations for large $|u|$ were studied by Olver \cite[Ch. 10, \S9]{Olver1997}.

\section*{Acknowledgement} The author was supported by a Premium Postdoctoral Fellowship of the Hungarian Academy of Sciences. The author wish to thank Takashi Aoki and Adri B. Olde Daalhuis for useful discussions and Kohei Iwaki for drawing his attention to the paper \cite{Kamimoto2011}. The author also thanks the referee for helpful comments and suggestions for improving the presentation.

\appendix \section{}\label{Appendix}

The evaluation of the coefficients $\mathsf{A}_n^\pm (\xi)$ via the recursions \eqref{wkbcoeffrec} may be quite cumbersome in practice as they involve nested integrations. (See the first example in Section \ref{Section6}, above.) In this appendix, we present a different method for the computation of these coefficients which avoids nested integrations. We begin with formal solutions of the alternative form
\begin{equation}\label{altwkb}
W^\pm  (u,\xi ) = \exp \!\left( \pm u \xi  \pm \frac{1}{2}\int_{}^\xi \phi (t)\d t  + \sum\limits_{n = 1}^\infty \frac{\mathsf{E}_n^\pm(\xi )}{u^n} \right) .
\end{equation}
The coefficients $\mathsf{E}_n^\pm(\xi )$ are found by substitution of \eqref{altwkb} into \eqref{Eq} and equating like powers of $u$. In this way, we find that
\begin{equation}\label{Ecoeff}
\mathsf{E}_n^ \pm  (\xi ) = \int_{}^\xi  \mathsf{F}_n^ \pm  (t)\d t,
\end{equation}
where
\[
\mathsf{F}_1^ \pm  (\xi ) =  - \frac{1}{4}\phi '(\xi ) \mp \frac{1}{8}\phi ^2 (\xi ) \pm \frac{1}{2}\psi (\xi ),
\]
and
\[
\mathsf{F}_{n + 1}^ \pm  (\xi ) =  - \frac{1}{2}\phi (\xi )\mathsf{F}_n^ \pm  (\xi ) \mp \frac{1}{2}\frac{\d \mathsf{F}_n^ \pm  (\xi )}{\d\xi } \mp \frac{1}{2}\sum\limits_{k = 1}^{n - 1} \mathsf{F}_k^ \pm  (\xi )\mathsf{F}_{n - k}^ \pm  (\xi )
\]
for $n\geq 1$, with the understanding that empty sums are zero. The constants of integration in \eqref{altwkb} and \eqref{Ecoeff} are arbitrary. Comparing \eqref{altwkb} with \eqref{WKB1}, it is seen that the following pair of formal relations hold between the coefficients $\mathsf{A}_{n}^\pm(\xi )$ and $\mathsf{E}_{n}^\pm(\xi )$:
\[
\exp \!\left( \sum\limits_{n = 1}^\infty \frac{\mathsf{E}_n^ \pm  (\xi )}{u^n } \right) = 1 + \sum\limits_{n = 1}^\infty  \frac{\mathsf{A}_n^ \pm  (\xi )}{u^n} .
\]
Application of Ex. 8.3 of Olver \cite[p. 22]{Olver1997} then leads to the recurrence relations
\begin{equation}\label{Arecalt}
\mathsf{A}_{n}^\pm(\xi ) = \mathsf{E}_{n}^\pm(\xi ) + \frac{1}{n}\sum\limits_{k = 1}^{n - 1} k\mathsf{E}_{k}^\pm(\xi )\mathsf{A}_{n - k}^\pm(\xi )
\end{equation}
for $n\geq 1$. Once the constants of integration in \eqref{Ecoeff} are fixed, the $\mathsf{A}_{n}^\pm(\xi )$'s are uniquely determined by \eqref{Arecalt}. For example, to obtain the coefficients generated by the recursive formulae \eqref{Arec}, the lower limits of integration in \eqref{Ecoeff} are taken as $\mp \infty  + \im\Im \xi$.

We mention that computable error bounds for formal solutions of the form \eqref{altwkb} (in the special case that $\phi(\xi) \equiv 0$) have recently been given by Dunster \cite{Dunster2020}.

\smallskip 

\end{document}